\documentclass[11pt]{article}
\usepackage{amsmath, amsthm, amssymb, epsfig, sectsty, graphicx, float, subfig, url,tabularx}
\usepackage{algorithm, algorithmicx, algpseudocode}
\usepackage[hypertexnames=false,hyperfootnotes=false]{hyperref}
\usepackage[hmargin=1in,vmargin=1in]{geometry}
\usepackage{blindtext}
\usepackage{graphicx}

\theoremstyle{plain}
\newtheorem{theorem}{Theorem}[section]
\newtheorem{lemma}[theorem]{Lemma}

\theoremstyle{definition}
\newtheorem{definition}[theorem]{Definition}

\newtheorem{prop}[theorem]{Propostion}

\makeatletter \renewenvironment{proof}[1][\proofname]{\par\pushQED{\qed}\normalfont\topsep6\p@\@plus6\p@\relax\trivlist\item[\hskip\labelsep\bfseries#1\@addpunct{.}]\ignorespaces}{\popQED\endtrivlist\@endpefalse} \makeatother

\floatstyle{ruled}

\sectionfont{\large} \subsectionfont{\normalsize}%
\numberwithin{equation}{section}%

\newcommand{\MySkip}[1]{}
\mathchardef\mhyphen="2D
\usepackage{color}

\numberwithin{equation}{section}%

\newcommand{\mb}[1]{\ensuremath{\boldsymbol{#1}}}

\mathchardef\mhyphen="2D
\usepackage{footnote}
\makesavenoteenv{tabular}
\makesavenoteenv{table}

\newcommand*{\red}{\textcolor{black}}

\title{ A Tractable Approach for designing Piecewise Affine Policies in Two-stage Adjustable Robust Optimization}

\author{Aharon Ben-Tal  \thanks{Industrial Engineering and Management, Technion - Israel Institute of Technology and CentER, Tilburg University, Netherlands Email: \texttt{abental@ie.technion.ac.il}.} \and Omar El Housni$\;^{\dagger}$ \and Vineet Goyal \thanks{Industrial Engineering and Operations Research, Columbia University, Email: \texttt{\{oe2148,vg2277\}@columbia.edu}. }}

\begin{document}
\maketitle
\begin{abstract}
We consider the problem of designing piecewise affine policies for
two-stage adjustable robust linear optimization problems under right-hand side
uncertainty. It is well known that a piecewise affine policy is optimal although the number of pieces can be
exponentially large. A significant challenge in designing a practical piecewise affine policy is constructing good pieces of the uncertainty set. Here we address this challenge by introducing a new framework in which the uncertainty set is ``approximated'' by a ``dominating'' simplex. The corresponding policy is then based on a mapping from the uncertainty set to the simplex. \red{Although our piecewise affine policy has exponentially many pieces, it can be computed efficiently by solving a compact linear program given the dominating simplex. Furthermore, we can find the dominating simplex in a closed form if the uncertainty set satisfies some symmetries and can be computed using a MIP in general. The performance of our policy is significantly better than the affine policy for many important uncertainty sets, such as ellipsoids and norm-balls, both theoretically and numerically.} For instance, for hypersphere uncertainty set, our piecewise affine policy can be computed by an LP and gives a $O(m^{1/4})$-approximation whereas the affine policy requires us to solve a second order cone program and has a worst-case performance bound of $O(\sqrt m)$.

\end{abstract}

\section{Introduction}

Addressing uncertainty in problem parameters in an optimization problem is a fundamental challenge in most real world problems where decisions often need to be made in the face of uncertainty. Stochastic and robust optimization are two approaches that have been studied extensively to handle uncertainty. In a stochastic optimization framework, uncertainty is modeled using a probability distribution and the goal is to optimize an expected objective~\cite{Dantzig55}. We refer the reader to Kall and Wallace~\cite{KW94}, Prekopa~\cite{Prekopa95}, Shapiro~\cite{Shapiro08}, Shapiro et al.~\cite{SDR09} for a detailed discussion on stochastic optimization. While it is a reasonable approach in certain settings, it is intractable in general and suffers from the ``curse of dimensionality''. Moreover, in many applications, we may not have sufficient historical data to estimate a joint probability distribution over the uncertain parameters.

Robust optimization is another paradigm where we consider an adversarial model of uncertainty using an uncertainty set and the goal is to optimize over the worst-case realization from the uncertainty set. This approach was first introduced by Soyster~\cite{SA73} and has been extensively studied in recent past. We refer the reader to Ben-Tal and Nemirovski~\cite{BN98,BN99,Ben-Tal02}, El Ghaoui and Lebret~\cite{EL97}, Bertsimas and Sim~\cite{BS03,BS04}, Goldfarb and Iyengar~\cite{GI03}, Bertsimas et al.~\cite{BBC08} and Ben-Tal et al.~\cite{BNE10} for a detailed discussion of robust optimization. Robust optimization leads to a tractable approach where an optimal static solution can be computed efficiently for a large class of problems. \red{Moreover, in many cases, designing an uncertainty set is significantly less challenging than estimating a joint probability distribution for high-dimensional uncertainty.} However, computing an optimal adjustable (or dynamic) solution for a multi-stage problem is generally hard even in the robust optimization framework.

In this paper, we consider  two-stage adjustable robust (AR) linear optimization problems with covering constraints and uncertain right-hand side. In particular, we consider the following model:

\medskip

\noindent
$\Pi_{\sf AR}({\cal U})$:
\begin{equation}\label{eqn:ar}
\begin{aligned}
z_{\sf AR}({\cal U}) = \min \; & \mb{c}^T \mb{x} + \max_{\mb{h}\in {\cal U}} \min_{\mb{y}(\mb{h})} \mb{d}^T \mb{y}(\mb{h}) \\
&  \mb{A}\mb{x} + \mb{B}\mb{y}(\mb{h}) \; \geq \; \mb{h}   \; \; \; \forall \mb h \in {\cal U}  \\
 & \mb{x} \; \in \; {\mathbb R}^{n_1}_+\\
 & \mb{y}(\mb{h})  \; \in \; {\mathbb R}^{n_2}_+,
\end{aligned}
\end{equation}
where $\mb{A} \in {\mathbb R}_+^{m {\times} n_1}, \mb{c}\in {\mathbb R}^{n_1}_+, \mb{d}\in {\mathbb R}^{n_2}_+, \mb{B} \in {\mathbb R}^{m {\times} n_2}$, and ${\cal U}\subseteq{\mathbb R}^m_+$ is the uncertainty set. The goal in this problem is to select the first-stage decision $\mb{x}$, and the second-stage recourse decision,  $\mb y(\mb{h})$, as a function of the uncertain right hand side realization, $\mb{h}$ such that the worst-case cost over all realizations of $\mb h \in {\cal U}$ is minimized. We assume without loss of generality that $n_1 = n_2 = n$ and that the uncertainty set ${\cal U}$ satisfies the following assumption.

\vspace{2mm}
\noindent {\bf Assumption 1}. ${\cal U}\subseteq [0,1]^m$  is convex, full-dimensional with $\mb{e}_i \in {\cal U}$ for all $i =1,\ldots,m$, and  {\em down-monotone}, i.e.,  $\mb{h}\in{\cal U}$ and $\mb{0}\leq \mb{h}' \leq\mb{h}$ implies that $\mb{h}' \in{\cal U}$.

\vspace{2mm}
\red{ We would like to emphasize that the above assumption can be made without loss of generality since we can appropriately scale the uncertainty set, and consider a down-monotone completion, without affecting the two-stage problem~\eqref{eqn:ar}.} Note that in the model $\Pi_{\sf AR}({\cal U})$ the objective coefficients $\mb{c}$, $\mb{d}$, the first-stage constraint matrix $\mb{A}$, and the decision variables $\mb{x}, \mb{y}(\mb{h})$ are all non-negative. This is restrictive as compared to general two-stage linear programs but the above model still captures many important applications including set cover, facility location and network design problems under uncertain demand. Here the right-hand side, $\mb h$ models the uncertain demand and the covering constraints capture the requirement of satisfying the uncertain demand.

The worst case scenario of problem \eqref{eqn:ar} occurs on extreme points of ${\cal U}$. Therefore, given an explicit list of the extreme points of the uncertainty set ${\cal U}$, the adjustable robust optimization problem~\eqref{eqn:ar} can be solved efficiently by including the second-stage decisions and the covering constraints only for the extreme points of ${\cal U}$. \red{Some approaches have been developed to generate dynamically the required extreme points, e.g. Zeng and Zhao \cite{zeng2011solving}, Ayoub and Poss \cite{ayoub2016decomposition}.   } However, in general the adjustable robust optimization problem \eqref{eqn:ar} is intractable; for example, when the number of extreme points is large or due to other structural complexities of ${\cal U}$. \red{ In fact, Feige et al.~\cite{FJMM07} show that problem $\Pi_{\sf AR}({\cal U})$ is hard to approximate within any factor that is better than $\Omega(\log m)$, even in the case of budget uncertainty set and $\mb A$,$\mb B$ being $0$-$1$ matrices}. This motivates us to consider approximations for the problem. Static robust and affinely adjustable solution approximations have been studied in the literature for this problem. In a static robust solution, we compute a single optimal solution $(\mb{x},\mb{y})$ that is feasible for all realizations of the uncertain right hand side. Bertsimas et al.~\cite{BGS10}  relate the performance of static solution to the symmetry of the uncertainty set and show that it provides a good approximation to the adjustable problem if the uncertainty is close to being centrally symmetric. \red{However, the performance bound of static solutions can be arbitrarily large for a general convex uncertainty set with the worst case performance being $\Omega (m)$}.
El Housni and Goyal \cite{elhousni2015piecewise} consider piecewise static policies for two-stage adjustable robust problem with uncertain constraint coefficients. These are a generalization of static policies where the uncertainty set is divided into several pieces and a static solution specified for each piece. However, they show that, in general, there is no piecewise static policy with a polynomial number of pieces that has a significantly better performance than an optimal static policy.

Ben-Tal et al.~\cite{Ben-Tal04} introduce an affine adjustable solution (also known as affine policy) to approximate adjustable robust problems. Affine policy restricts the second-stage decisions, $\mb{y}(\mb{h})$ to being an affine function of the uncertain right-hand side $\mb{h}$, i.e.,  $\mb{y}(\mb{h})=\mb{P}\mb{h}+\mb{q}$ for some $\mb{P}\in{\mathbb R}^{n\times m}$ and $\mb{q}\in{\mathbb R}^m$, which are decision variables on top of $\mb x \in {\mathbb R}_+^n$. An optimal affine policy can be computed efficiently for a large class of problems and has a strong empirical performance. For a class of multistage problems where there is a single uncertain parameter in each period, Bertsimas et al.~\cite{BIP09} and Iancu et al.~\cite{ISS13} show that affine policies are optimal. Bertsimas and Goyal~\cite{BG10} show that affine policies are optimal if the uncertainty set ${\cal U}$ is a simplex. They prove a worst case bound of $O(\sqrt{m})$  on the performance of affine policy for general uncertainty sets. Moreover, they show that this bound is tight for an uncertainty set quite analogous to the intersection of the unit $\ell_2$-norm ball and the non-negative orthant, i.e.,

\red{\begin{equation} \label{def:sphere}
{\cal U}=\{\mb{h}\in{\mathbb R}^m_+\;|\;||\mb{h}||_2\leq 1\}.
\end{equation}}
Bertsimas and Bidkhori~\cite{BB15} provide improved approximation bounds for affine policies for~$\Pi_{\sf AR}({\cal U})$ that depend on the geometric properties of the uncertainty set. \red{More general decision rules have been considered in the literature and tested numerically;
extended affine decision rules (Chen et al. \cite{chen2008linear}), binary decision rules (Bertsimas and Georghiou \cite{bertsimas2015design}) and adjustable solutions via iterative splitting of uncertainty sets, (Postek and Den Hertog \cite{postek2016multistage}). More recently, Bertsimas and Dunning~\cite{BD15} give an MIP-based algorithm to adaptively partition the uncertainty set.  However, no theoretical guarantees on the performance, or the number of partitions, are known.
}

Piecewise affine policies (PAP) have been studied earlier. In a PAP, we consider pieces ${\cal U}_i, i\in[k]$ of ${\cal U}$ such that ${\cal U}_i \subseteq {\cal U}$ and ${\cal U}$ is covered by the union of all pieces. For each ${\cal U}_i$, we have an affine solution $\mb y ( \mb h)$ where $ \mb h \in {\cal U}_i$. PAP are significantly more general than static and affine policies.  For problem  $\Pi_{\sf AR}({\cal U})$, with ${\cal U}$ being a polytope, a PAP is known to be optimal.  However, the number of pieces can be exponentially large.  Moreover, finding the optimal pieces is, in general, an intractable task.  In fact, Bertsimas and Caramanis~\cite{bertsimas2010finite} prove that it is NP-hard to construct the optimal pieces, even for pieceiwse policies with two pieces, for two-stage robust linear programs.  

\subsection{Our Contributions}

Our main contributions in this paper are as follows.

\vspace{2mm}
\noindent {\bf New Framework for Piecewise affine policy}. We present a new framework to efficiently construct a ``good'' piecewise affine policy for the adjustable robust problem $\Pi_{\sf AR}({\cal U})$. As we mentioned earlier, one of the significant challenges in designing a piecewise affine policy arises from the need to construct ``good pieces" of the uncertainty set. We suggest a new approach where instead of  directly finding an explicit partition of $\cal U$, we approximate $\cal U$ with a ``simple'' set $\hat{\cal U}$ satisfying the following two properties:
\begin{enumerate}
\item  the adjustable robust problem~\eqref{eqn:ar} over $\hat {\cal U}$ can be solved efficiently,
\item $\hat{\cal U}$ ``dominates'' ${\cal U}$,  i.e.,  for any $\mb{h}\in{\cal U}$, there exists $\hat{\mb{h}}\in\hat{\cal U}$ such that $\mb{h}\leq\hat{\mb{h}}$.
\end{enumerate}
Using the uncertainty set $\hat{\cal U}$ instead of ${\cal U}$,
the domination property of $\hat{\cal U}$ preserves the feasibility of the
adjustable robust problem. Specifically, we choose $\hat{\cal U}$ to be a simplex dominating ${\cal U}$. Therefore, the adjustable robust problem~\eqref{eqn:ar} over $\hat{\cal U}$ can be solved efficiently since $\hat{\cal U}$ only has $m+1$ extreme points. \red{We construct a piecewise affine mapping between the uncertainty set ${\cal U}$ and the dominating set $\hat{\cal U}$, i.e. we use a piecewise affine function to map each point $\mb h \in {\cal U}$ to a point $\hat{\mb h}$ that dominates $\mb h$. This mapping leads to
our piecewise affine policy which is constructed from an optimal adjustable solution over $\hat{\cal U}$.} We show that the performance of our policy is significantly better than the affine policy for many important uncertainty sets both theoretically and numerically.

We elaborate on the two ingredients of designing our piecewise affine policy below, namely, constructing $\hat{\cal U}$ and the corresponding piecewise map below.


\begin{itemize}

\item[a)] {\bf Constructing a dominating uncertainty set}. Our framework is based on choosing an appropriate {\em dominating simplex } $\hat{\cal U}$ based on the geometric structure of ${\cal U}$. Specifically,  $\hat{\cal U}$ is taken to be a simplex of the following form
$$ \hat{\cal U} = \beta \cdot {\sf conv } \left( \mb e_1,\ldots,\mb e_m, \mb v \right), $$
where $\beta >0$ and $ \mb v \in {\cal U}$ are chosen appropriately so that $\hat{\cal U}$ dominates ${\cal U}$. Solving the adjustable robust problem over $\hat{\cal U}$ gives a feasible solution for problem $\Pi_{\sf AR}({\cal U})$ due to the domination property. Moreover, the optimal adjustable solution over $\hat{\cal U}$ gives a $\beta$-approximation for problem $\Pi_{\sf AR}({\cal U})$, since $\hat{\cal U} = \beta \cdot {\sf conv } \left( \mb e_1,\ldots,\mb e_m, \mb v \right) \subseteq \beta \cdot {\cal U}$. The approximation bound $\beta$  is related to a geometric {\em scaling factor} that represents the Banach-Mazur distance between ${\cal U}$ and $\hat{\cal U}$. We note that $\hat {\cal U}$ does not necessarily contain ${\cal U}$.
\ \\

\item[b)] {\bf The piecewise affine mapping}. We employ the following piecewise affine mapping $\mb{ \hat{h}}(\mb h) = \beta \mb v + \left( \mb h - \beta \mb v\right)^+$ that maps any $\mb h \in {\cal U}$ to a dominating point $\hat{\mb h}$ such that $\mb h \leq \hat{\mb h}$.
For any $\mb{h} \in {\cal U}$, $\mb{\hat h}(\mb{h})$ is contained in the down-monotone completion of $2 \cdot \hat{\cal U}$. The piecewise affine policy is based on the above piecewise affine mapping and gives a $2\beta$-approximation for problem $\Pi_{\sf AR}({\cal U})$. In this policy, $ \beta \mb v$ is covered by the static component and  $\left( \mb h - \beta \mb v\right)^+$ is covered by the piecewise linear component of our policy. This is quite analogous to {\em threshold policies} that are widely used in dynamic optimization.  Note that $\hat {\mb h}$ does not necessarily belong to $\hat{\cal U}$ but is contained in the down-monotone completion of $2 \cdot \hat{\cal U}$ and therefore, we get an approximation factor of $2 \beta$ instead of $\beta$. We can construct a set-dependent piecewise affine map between ${\cal U}$ and $\hat{\cal U}$ that allows us to construct a piecewise affine policy with a performance bound of $\beta$. \red{This bound $\beta$ is not affected by the scaling introduced in Assumption 1.}

\end{itemize}

Given the dominating set, $\hat{\cal U}$, our piecewise affine policy can be computed efficiently;  in fact, it can be computed even more efficiently than an affine solution over ${\cal U}$ in many cases because the adjustable problem over $\hat{\cal U}$ is a simple LP with only $m+1$ constraints while the affine problem over ${\cal U}$  is a general convex program for general convex uncertainty sets.



\vspace{3mm}
\noindent {\bf \red{Results for Scaled Permutation Invariant (SPI) Sets}}. \red{The uncertainty set ${\cal U}$ is SPI if ${\cal U} = \mbox{diag } ( \mb \lambda) \cdot {\cal V}$ where $\mb \lambda \in {\sf R}_+^m$ and ${\cal V}$ is an {\it invariant set}, i.e., if $\mb{v} \in {\cal V}$, then any permutation of the components of $\mb{v}$ are also in ${\cal V}$. SPI sets include ellipsoids, weighted norm-balls, intersection of norm-balls with budget uncertainty sets and more.  SPI sets are commonly used in Robust Optimization literature and in practice.}

We show that for SPI uncertainty set ${\cal U}$, it is possible to construct the dominating set $\hat{\cal U}$ and compute the scaling factor $\beta$.  In particular, we give an efficiently computable closed-form expression for $\beta$ and $\mb{v} \in {\cal U}$ that are needed to construct $\hat{\cal U}$.  Consequently, we can efficiently construct our piecewise affine decision rule, having a performance bound $2\beta$.

\red{Using this framework, we provide approximation bounds for the piecewise affine policy that are significantly better than those of the optimal affine policy in~\cite{BB15} for many SPI uncertainty sets}. For instance, we show that our policy gives a  $O (m^{1/4})$-approximation for the two-stage adjustable robust problem~\eqref{eqn:ar} with hypersphere uncertainty set as in~\eqref{def:sphere}, compared to the affine policy in \cite{BB15} that has an approximation bound of $O(\sqrt{m})$. More generally, the performance bound for our policy for the $p$-norm ball is $ O ( m^{\frac{p-1}{p^2}} )$ as opposed to $ O  ( m^\frac{1}{p} ) $ given by the affine policy in~\cite{BB15}~\footnote{\textbf{Remark.} We note that in~\cite{BB15}, in Tables 1 and 2, there is a typo in the performance bound for affine policies for $p$-norm balls. According to Theorem 3 in~\cite{BB15}, the bound should be
$$\frac{m^{\frac{p-1}{p}}+m}{m^{\frac{p-1}{p}}+m^{\frac{1}{p}}}= O \left(m^{\frac{1}{p}} \right),$$
instead of $\frac{m^{\frac{p-1}{p}}+m}{m^{\frac{1}{p}}+m}$ as mentioned in Table 2 in~\cite{BB15}).
}. Table~\ref{tab:results} summarizes the above comparisons.  \red{We also present computational experiments and observe that our policy also outperforms affine policy in computation time on several examples of uncertainty sets considered in our experiments including hypersphere, norm-balls and certain polyhedral uncertainty sets.  However, we would like to note that our piecewise affine policy does not a generalize affine policy and there are instances where affine policy performs better than our policy. For instance, we observe in our computational experiments that the performance of affine policy is better than our policy for budget of uncertainty sets.}


\vspace{2mm}
\noindent {\bf Results for general uncertainty sets}. 
\red{While the dominating set $\hat{\cal U}$ is given in an efficiently computable closed-form expression for SPI sets, the construction of
$\hat{\cal U}$  for general uncertainty sets requires solving a sequence of MIPs   which is computationally much harder than for the case of SPI sets.
In Section 4, we give an algorithm for constructing the dominating set $\hat{\cal U}$, and a piecewise affine policy for general uncertainty set ${\cal U}$. Our framework is not necessarily computationally more appealing than computing optimal affine policies. However, we would like to note that in practice these MIPs can be solved efficiently.} Moreover, the construction of the dominating set $\hat{\cal U}$ is independent
of the parameters of the adjustable problem and depends only on the uncertainty set, ${\cal U}$. Therefore, $\hat{\cal U}$ can be computed offline and then used to construct the piecewise affine policy efficiently.

We show that our policy gives a $O(\sqrt m)$-approximation for general uncertainty sets which is same as the worst-case performance bound for affine policy. We also show that the bound of $O(\sqrt m)$ is tight. In particular, for the budget uncertainty set
\[{\cal U}= \left\{ \mb h \in \mathbb R^{m}_+ \;  \bigg\vert \;     \sum_{i=1}^m h_i = \sqrt{m} , \; 0 \leq h_i \leq 1 \ \;  \forall i \in [m]       \right\},\]
the performance bound of our piecewise affine policy is $\Theta(\sqrt m)$. Furthermore, the bound of $\Theta(\sqrt m)$ holds even if we consider dominating sets with a polynomial number of extreme points that are significantly more general than a simplex.  While this example shows that the worst-case performance of our policy is the same as the worst-case performance of the affine policy, we would like to emphasize that our policy still gives a significantly better approximation than affine policies for many important uncertainty sets, and does so in a fraction of computing time (see Section 6.2).

\begin{table}[h]
\centering
\resizebox{\columnwidth}{!}{%
\begin{tabular} {|c|l|c|c|}
  \hline
No. & Uncertainty set  &Bounds in ~\cite{BB15} & Our Bounds \\
  \hline
1 & $\left\{ \mb h \in \mathbb R^{m}_+ \;  \big\vert \; \| \mb h\|_2 \leq 1 \right\}$& $O \left( \sqrt{m} \right)$ & $ O \left( m^{\frac{1}{4}} \right)$  \\
   \hline
   2 & \red{$\left\{ \mb h \geq \mb 0 \; \left | \;  \sum_{i=1}^m r_i h_i^ 2 \leq 1 \right.\right\}$}& \red{$O \left( \sqrt{m} \right)$} & \red{$ O \left( m^{\frac{1}{4}} \right)$}  \\
  \hline
   3 & $\left\{\mb{h}\in{\mathbb R}^{m}_+\;|\;   \mb{h}^T \mb \Sigma \mb{h}       \leq 1\right\}$&---& $ O  \left( m^{\frac{2}{5}}\right) $\\
  \hline
  4 & $\left\{ \mb h \in \mathbb R^{m}_+ \;  \big\vert \; \| \mb h\|_p \leq 1 \right\}$ &$O \left( m^{\frac{1}{p}}\right)$ & $O\left(m^{\frac{p-1}{p^2}}\right)$  \\
  \hline
  5 & $\left\{ \mb h \in \mathbb R^{m}_+ \;  \big\vert \; \| \mb h\| _p \leq 1,  \;  \| \mb h\| _q \leq r  \right\}$ & $ O \left( r^{-1} m^{\frac{1}{q}}\right)  $&    $O \left( \min \left( r^{\frac{1-p}{p}}m^{\frac{p-1}{pq}} , r^{\frac{1}{q}}m^{\frac{q-1}{q^2}} \right) \right) $  \\
  \hline
  6 & $\left\{ \mb h \in [0,1]^{m} \;  \big\vert \; \sum_{i=1}^m h_i \leq k \right\}$&  $ O \left(\frac{k^2+mk}{k^2+m}   \right)$ & $ O \left( \min \left( k, \frac{m}{k}\right) \right) $\\
  \hline
\end{tabular}%
}
\caption{Comparison with performance bounds for affine policies in Bertsimas and Bidkhori \cite{BB15}. The ellipsoid in Example 3 is assumed to be a permutation invariant set. There is no specialized bound for this Ellipsoid in \cite{BB15}. For intersection of norm-balls (Example 4 in the table), we assume  $ m^{\frac{1}{q}-\frac{1}{p}} \geq r \geq 1 $.}\label{tab:results}
 \end{table}


\vspace{3mm}\noindent\textbf{Outline.} In Section~\ref{sec:framework}, we present the new framework for approximating the two-stage adjustable robust problem~\eqref{eqn:ar} via dominating uncertainty sets and constructing piecewise affine policies. 
\red{In Section~\ref{sec:perm-inv}, we provide improved approximation bounds for~\eqref{eqn:ar} for scaled permutation invariant sets.} We present the  case of general uncertainty sets in Section~\ref{sec:gen}.
In Section~\ref{sec:worst-case}, we present a family of lower-bound instances where our piecewise affine policy has the worst performance bound and finally in Section~\ref{sec:sim}, we present a computational study to test our policy and compare it to an affine policy over ${\cal U}$.


\section{A new framework for piecewise affine policies}~\label{sec:framework}
We present a piecewise affine policy to approximate the two-stage adjustable robust problem~\eqref{eqn:ar}. Our policy is based on approximating the uncertainty set ${\cal U}$ with a simple set $\hat{\cal U}$ such that the adjustable  problem~\eqref{eqn:ar} can be efficiently solved over $\hat{\cal U}$. In particular, we select $\hat{\cal U}$  such that it  {\em dominates} ${\cal U}$ and it is {\em close} to ${\cal U}$. We make these notions precise with the following definitions.

\begin{definition} {\bf (Domination)}\label{Domination}
Given an uncertainty set ${\cal U}\subseteq{\mathbb R}^{m}_+$, $\hat{{\cal U}}\subseteq{\mathbb R}^{m}_+$ dominates ${\cal U}$ if for all
$\mb{h}\in {\cal U}$, there exists $\hat{\mb{h}}\in\hat{\cal U}$ such that $\hat{\mb{h}}\geq\mb{h}$.
\end{definition}

\begin{definition} {\bf (Scaling factor)}\label{def:scaling-factor}
Given a full-dimensional uncertainty set ${\cal U}\subseteq{\mathbb R}^{m}_+$ and $\hat{\cal U}\subseteq{\mathbb R}^{m}_+$ that dominates ${\cal U}$. We define the scaling factor $\beta ({\cal U},\hat{{\cal U}})$  as following
\[\beta ({\cal U},\hat{{\cal U}})= \min \left\{ \beta >0 \; |\;\hat{\cal U} \subseteq  \beta \cdot {\cal U} \right\} .\]
\end{definition}

For the sake of simplicity, we denote the scaling factor $\beta ({\cal U},\hat{{\cal U}})$ by $\beta$ in the rest of this paper. The scaling factor always exists since ${\cal U}$ is full-dimensional. Moreover, it is greater than one because $\hat{\cal U}$ dominates ${\cal U}$.
Note that the dominating set $\hat{\cal U}$ does not necessarily contain ${\cal U}$. We illustrate this in the following example.

\vspace{2mm}
\noindent {\bf Example}. Consider the uncertainty set  $ {\cal U} $ defined in \eqref{def:sphere} which is the intersection of the unit $\ell_2$-norm ball and the non-negative orthant. We show later in this paper (Proposition \ref{prop:sphere}) that the simplex $\hat{\cal U} $  dominates ${\cal U}$ where
\begin{equation}\label{def:dom:sphere}
\hat{\cal U}= m^{\frac{1}{4}}\cdot {\sf conv } \left( \mb e_1,\ldots,\mb e_m, \frac{1}{\sqrt{m}}\mb e \right).
\end{equation}
Figures~\ref{fig:sph} and~\ref{fig:doma} illustrate the sets ${\cal U} $ and $\hat{\cal U}$ for $m=3$. Note that $\hat{\cal U}$ does not contain ${\cal U}$ but only dominates ${\cal U}$. This is an important property in our framework.
\begin{figure}[H]
\centering
\minipage{0.4\textwidth}
		\includegraphics[scale=0.35]{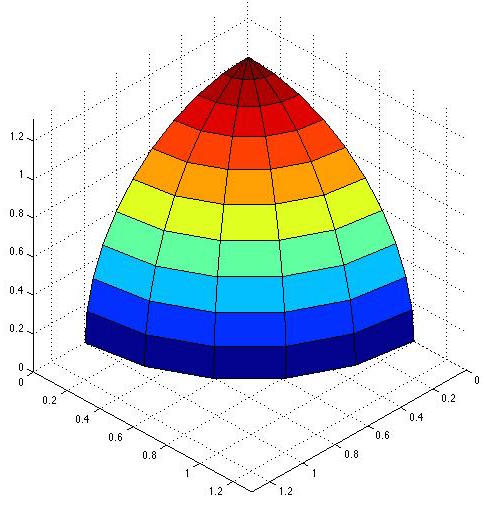}
 		 \caption{The uncertainty set \eqref{def:sphere}  }
 		 \label{fig:sph}
\endminipage\hfill 
\minipage{0.5\textwidth}
		\includegraphics[scale=0.48]{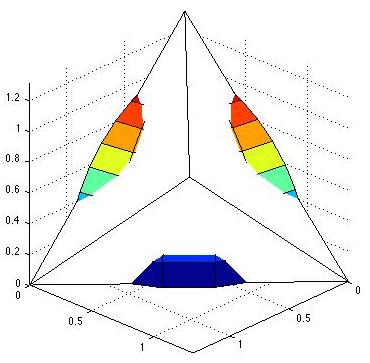}
   	 	\caption{ The dominating set $\hat{\cal U}$ \eqref{def:dom:sphere}}
   	 	\label{fig:doma}
\endminipage\hfill
\end{figure}

 The following theorem shows that solving the adjustable problem over the set $\hat{\cal U}$ gives a $\beta$-approximation to the two-stage adjustable robust problem~\eqref{eqn:ar}.

\begin{theorem} \label{thm:ineq}
\red{Consider an uncertainty set ${\cal U}$ that verifies Assumption $1$ and $\hat{\cal U}\subseteq{\mathbb R}^m_+$ that dominates ${\cal U}$.} Let $\beta$ be the scaling factor of $({\cal U},\hat{\cal U})$. Moreover, let $z_{\sf AR}({\cal U})$ and $z_{\sf AR}(\hat{\cal U})$ be the optimal values for~\eqref{eqn:ar} corresponding to  $\cal U$ and $\hat{\cal U}$, respectively. Then,
\[z_{\sf AR}({\cal U}) \leq z_{\sf AR}(\hat{\cal U}) \leq \beta \cdot z_{\sf AR}({\cal U}).\]
\end{theorem}

The proof of Theorem \ref{thm:ineq} is presented in Appendix \ref{apx-proofs:thm:ineq}.

\subsection{Choice of $\hat{\cal U}$}

Theorem~\ref{thm:ineq} provides a new framework for approximating the two-stage adjustable robust problem $\Pi_{\sf AR}(\cal U)$~\eqref{eqn:ar}. \red{Note that we require  $\hat{\cal U}$ to be such that it dominates ${\cal U}$ and that $\Pi_{\sf AR}(\hat{\cal U})$ can be solved efficiently over $\hat{\cal U}$}. In fact, the latter is satisfied if the number of extreme points of $\hat{\cal U}$ is small and is explicitly given (typically polynomial of $m$). In our framework, we choose the dominating set to be a simplex of the following form
\begin{equation}\label{eq:simplex}
\hat{\cal U} = \beta \cdot {\sf conv } \left( \mb e_1,\ldots,\mb e_m, \mb v \right),
\end{equation}
for some $ \mb v \in {\cal U}$. The coefficient $\beta$ and $\mb{v} \in {\cal U}$ are chosen such that $\hat {\cal U}$ dominates $ {\cal U }$. 
For a given $\hat{\cal U}$ (i.e.,  $\beta$ and $\mb v \in {\cal U}$), the adjustable robust problem, $\Pi_{\sf AR}(\hat{\cal U})$~\eqref{eqn:ar} can be solved efficiently as it can be reduced to the following LP:
\[\begin{aligned}
z_{\sf AR}(\hat{\cal U}) = \min  \;& \mb{c}^T\mb{x}+z \\
& z\geq\mb{d}^T\mb{y}_i,\;\forall i\in[m+1] \\
& \mb{A}\mb{x}+\mb{B}\mb{y}_i\geq \beta \mb{e}_i,\;\forall i\in[m] \\
& \mb{A}\mb{x}+\mb{B}\mb{y}_{m+1}\geq \beta \mb v \\
& \mb{x}\in {\mathbb R}^{n}_+, \;  \mb{y}_i\in{\mathbb R}^{n}_+, \;\forall i\in[m+1] .\\
\end{aligned}\]

\subsection{Mapping points in ${\cal U}$ to dominating points }

Consider the following  piecewise affine mapping for any  $\mb h \in {\cal U}$:

\begin{equation} \label{eq:mapping}
\forall \mb h \in {\cal U}, \qquad \mb{ \hat{h}}(\mb h) = \beta \mb v + ( \mb h - \beta \mb v )_+.
\end{equation}
We show that this maps any $\mb h \in {\cal U}$ to a dominating point contained in the down-monotone completion of $2 \cdot \hat{\cal U}$. First, the following structural result is needed.

\begin{lemma}{\bf (Structural Result)} \label{lem:ineq}
Consider an uncertainty set ${\cal U}$ that verifies Assumption $1$.

{\bf $a)$}
Suppose there exists $ \beta $ and $ \mb v \in {\cal U}$ such that
$ \;  \hat{\cal U} = \beta \cdot {\sf conv } \left( \mb e_1,\ldots,\mb e_m, \mb v \right)$ dominates $\cal U $. Then,
\begin{equation} \label{ineq:beta-v}
\frac{1}{\beta} \sum_{i=1}^m \left( h_i -\beta v_i  \right)^{+} \leq 1, \; \forall \mb h \in {\cal U}.
\end{equation}

{\bf $b)$} \red{Moreover, if there exists $\beta $ and $ \mb v \in {\cal U}$ satisfying~\eqref{ineq:beta-v}.
 Then,\\  $2\beta \cdot {\sf conv } \left( \mb e_1,\ldots,\mb e_m, \mb v \right)$ dominates ${\cal U} $.}
\end{lemma}
The proof of Lemma \ref{lem:ineq} is presented in Appendix \ref{apx-proofs:lem:ineq}.

The following lemma shows that the mapping in~\eqref{eq:mapping} maps any $\mb h \in {\cal U}$ to a dominating point that belongs to the \red{down-monotone} completion of $2 \cdot \hat{\cal U}$.
 \begin{lemma} \label{lem:2U}
 \red{For all $ \mb h \in {\cal U}$, $  \mb{ \hat{h}}(\mb h) $ as defined in \eqref{eq:mapping} is a dominating point that belongs to the down-monotone completion of $2 \cdot \hat{\cal U}$}.
 \end{lemma}
 \begin{proof}
  It is clear that $\mb{ \hat{h}}(\mb h) $ dominates $ \mb h$ because $\mb{ \hat{h}}(\mb h) \geq \beta \mb v + ( \mb h - \beta \mb v ) = \mb h $.
  Moreover,  for all $ \mb h \in {\cal U}$, we have
\red{
\begin{align*}
\hat{\mb h}(\mb h) &= \beta \mb v + \frac{1}{\beta} \sum_{i=1}^m ( h_i - \beta v_i )^+ \beta \mb e_i \\
& \leq \underbrace{ \beta \mb v }_{\in \hat{\cal U} }+ \underbrace{\frac{1}{\beta} \sum_{i=1}^m ( h_i - \beta v_i )^+ \beta \mb e_i + ( 1-  \frac{1}{\beta}\sum_{i=1}^m ( h_i - \beta v_i )^+ ) \beta \mb v }_{\in \hat{\cal U} } \in 2 \cdot \hat{\cal U}
\end{align*}}
where the inequality
$$ 1- \frac{1}{\beta} \sum_{i=1}^m (  h_i - \beta  v_i )^+ \geq 0.$$
follows from part $a)$ of Lemma \ref{lem:ineq}.
 \red{Therefore, $  \mb{ \hat{h}}(\mb h) $ belongs to the down-monotone completion of $2 \cdot \hat{\cal U}$}.
   \end{proof}

\subsection{Piecewise affine policy }
We construct a \emph{piecewise affine policy over} ${\cal U}$ from the optimal solution of  $\Pi_{\sf AR}(\hat{\cal U})$ based on the piecewise affine mapping in \eqref{eq:mapping}. 
Let $\mb{\hat{x}}, \mb{\hat{y}(\hat{h})}$ for $\mb{\hat{h}} \in \hat{\cal U }$ be an optimal solution of $\Pi_{\sf AR}(\hat{\cal U})$. Since $\hat{\cal U}$ is a simplex, we can compute this efficiently.

\vspace{3mm}
\noindent {\bf The piecewise affine policy (PAP)}
\begin{equation} \label{picewise_sl}
\begin{split}
 \mb x  & =  2 \hat{\mb x}  \\
 \mb y( \mb h )  & = \frac{1}{\beta} \sum_{i=1}^m      \left( h_i -\beta v_i  \right)^+ \hat{\mb y }( \beta \mb e_i)+\hat{\mb y }(\beta \mb v), \qquad  \qquad \forall \mb h \in \cal U.
\end{split}
\end{equation}
The following theorem shows that the above PAP gives a $2\beta$-approximation for $\Pi_{\sf AR}({\cal U})$~\eqref{eqn:ar}.

\begin{theorem} \label{thm:cost}
Consider an uncertainty set ${\cal U}$ that verifies Assumption $1$ and  $\hat{\cal U} = \beta \cdot {\sf conv } \left( \mb e_1,\ldots,\mb e_m, \mb v \right)$ be a dominating set where $ \mb v \in {\cal U}$. The piecewise affine solution in \eqref{picewise_sl} is feasible and gives a $2\beta$-approximation for the adjustable robust problem, $\Pi_{\sf AR}({\cal U})$~\eqref{eqn:ar}.
\end{theorem}
\begin{proof}
First, we show that the policy \eqref{picewise_sl} is feasible. We have,
 \begin{align*}
\mb A \mb x + \mb B \mb y ( \mb h ) &=    2 \mb A \hat{\mb x}  + \mb B \left(\frac{1}{\beta} \sum_{i=1}^m\left( h_i -\beta v_i  \right)^+ \hat{\mb y }( \beta \mb e_i)+ \hat{\mb y }(\beta \mb v) \right)\\
&=  \left(  \mb A \hat{\mb x}+\mb B \hat{\mb y }( \beta\mb v ) \right)  +
 \mb A \hat{\mb x} + \frac{1}{\beta} \sum_{i=1}^m \left( h_i -\beta v_i  \right)^+ \mb B \hat{\mb y }( \beta  \mb e_i)\\
 &\geq \left(  \mb A \hat{\mb x}+\mb B \hat{\mb y }( \beta \mb v ) \right)  +
 \frac{1}{\beta} \sum_{i=1}^m \left( h_i -\beta v_i \right)^+ \left( \mb B \hat{\mb y }( \beta \mb e_i)  + \mb A \hat{\mb x} \right) \\
  &\geq \beta \mb v+
 \sum_{i=1}^m \left( h_i -\beta v_i  \right)^+  \mb e_i  \\
 &\geq \beta \mb v  +      \sum_{i=1}^m \left( h_i -\beta v_i \right)    \mb e_i = \mb h,\\
\end{align*}
where the first inequality follows from part $a)$ of Lemma \ref{lem:ineq} and the non-negativity of $\hat{\mb x}$ and $ \mb A$. The second inequality follows from the feasibility of $   \mb{\hat{x}},  \mb{\hat{y}(\hat{h})}$.

To compute the performance of~\eqref{picewise_sl}, we have for any $ \mb h \in {\cal U},$
\begin{align*}
\mb c^{T} \mb x + \mb d^{T} \mb y ( \mb h ) & = 2   \left(   \mb c^{T}    \hat{\mb x}  +    \mb d^{T}  \left( \frac{1}{2\beta} \sum_{i=1}^m \left( h_i -\beta v_i \right)^+ \hat{\mb y }( \beta \mb e_i)+\frac{1}{2}\hat{\mb y }( \beta \mb v)  \right)            \right)\\
& \leq 2  \left(   \mb c^{T}    \hat{\mb x}  +    \underset{ \mb {\hat h} \in   {\hat{\cal U}}} \max \; \mb d^T \hat{\mb y}( \hat{\mb h})  \left( \frac{1}{2 \beta}  \sum_{i=1}^m \left( h_i -\beta v_i  \right)^+ +\frac{1}{2} \right)            \right)\\
& \leq 2 \left(   \mb c^{T}    \hat{\mb x}  +   \underset{ \mb {\hat h} \in   {\hat{\cal U}}} \max \; \mb d^T \hat{\mb y}( \hat{\mb h})            \right)  \\
& = 2  \cdot z_{\sf{AR}}(  {\hat{\cal U}}),
\end{align*}
where the second last inequality follows from part $a)$ of Lemma \ref{lem:ineq}. From Theorem \ref{thm:ineq}, $z_{\sf{AR}}(  {\hat{\cal U}})  \leq   \beta \cdot      z_{\sf{AR}}\left( {\cal U}  \right)$. Therefore, the cost of the piecewise affine policy for any $\mb h \in {\cal U}$
  $$          \mb c^{T} \mb x + \mb d^{T} \mb y ( \mb h )  \leq  2 \beta \cdot z_{\sf{AR}}\left( {\cal U}  \right), $$
\red{which implies that the piecewise affine solution \eqref{picewise_sl} gives a $ 2 \beta$-approximation for the adjustable robust problem, $\Pi_{\sf AR}({\cal U})$~\eqref{eqn:ar}}.
 \end{proof}

The above proof shows that it is sufficient to find $\beta$ and $\mb v \in {\cal U}$ satisfying  \eqref{ineq:beta-v} in Lemma \ref{lem:ineq} to construct a piecewise affine policy that gives a $2 \beta$-approximation for (1.1). In particular, we summarize the main result in the following theorem.

\begin{theorem} \label{thm:addi}
Let the uncertainty set ${\cal U}$ satisfy Assumption $1$. Consider any $\beta $ and $ \mb v \in {\cal U}$ satisfying~\eqref{ineq:beta-v}. Then, the piecewise affine solution in \eqref{picewise_sl} gives a $2\beta$-approximation for the adjustable robust problem, $\Pi_{\sf AR}({\cal U})$~\eqref{eqn:ar}.
\end{theorem}

We would like to note that our piecewise affine policy in not necessarily an optimal piecewise policy. However, for a large class of uncertainty sets, we show that our policy is significantly better than affine policy and can even be computed more efficiently than an affine policy.





\section{\red{Performance Bounds for Scaled Permutation Invariant Sets}}~\label{sec:perm-inv}
\red{In this section, we present performance bounds of our policy for the class of scaled permutation invariant sets.  This class includes ellipsoids, weighted norm-balls, intersection of norm-balls and budget of uncertainty sets. These are widely used uncertainty sets in theory and in practice.}
\red{\begin{definition} {\bf Scaled Permutation Invariant Sets (SPI)}\label{perm}
\begin{enumerate}
\item ${\cal U}$ is a {\bf permutation invariant set} if $\mb{x}\in{\cal U}$ implies that for any permutation $\tau$ of $\{1,2,\ldots,m\}$, $\mb{x}^{\tau}\in{\cal U}$ where $x^\tau_i= x_{\tau(i)}$.
\item ${\cal U} $ is a {\bf scaled permutation invariant set} if there exists $\mb \lambda \in \mathbb{R}^m_+$ and $ {\cal V}$ a permutation invariant set such that $ {\cal U}= {\sf diag} (\mb \lambda) \cdot {\cal V}$
\end{enumerate}
\end{definition}
For a given SPI set ${\cal U}$, it is possible to scale the two-stage adjustable problem~\eqref{eqn:ar} and get a new problem where the uncertainty set is \emph{permutation invariant} (PI). Indeed, suppose $ {\cal U}= {\sf diag} (\mb \lambda) \cdot {\cal V}$ where $ {\cal V}$ is a permutation invariant set; by multiplying the constraint matrices $\mb A$ and $\mb B$ by ${\sf diag} (\mb \lambda) ^{-1}$, we get a new problem where the uncertainty set now is PI. The performance of our policy is not affected by this scaling. Therefore, without loss of generality, we consider in the rest of this section, the case of permutation invariant uncertainty sets.}


We first introduce some structural properties of PI sets. Let ${\cal U}$ be PI satisfying Assumption $1$. For all  $k={1,\ldots,m}$, let
\begin{equation} \label{eq:gamma(k)}
\gamma (k)=\frac{1}{k} \cdot \max \left\{ \sum_{i=1}^k h_i  \; \Big\vert\ \mb h \in {\cal U} \right\}.
\end{equation}
The coefficients, $\gamma(k)$ for all $k=1,\ldots,m$ affect the geometric structure of ${\cal U}$. In particular, we have the following lemma.

\begin{lemma} \label{lem:gamma(k)}
Le ${\cal U}$ be a permutation invariant set and $\gamma (\cdot)$ be as defined in \eqref{eq:gamma(k)}. Then,
$$ \gamma (k) \cdot \sum_{i=1}^k \mb e_i \in {\cal U}, \; \forall k=1,\ldots,m$$
\end{lemma}
We present the proof of Lemma~\ref{lem:gamma(k)} in Appendix \ref{apx-proofs:em:gamma(k)}. For the sake of simplicity, we denote $ \gamma (m)$ by $ \gamma $ in the rest of the paper. From the above lemma, we know that $\gamma \cdot \mb{e} \in {\cal U}$.

\subsection{Piecewise affine policy for Permutations Invariant Sets}
For any PI set ${\cal U}$, we consider the following dominating uncertainty set, $\hat{\cal U}$ of the form~\eqref{eq:simplex} with $\mb v = \gamma \mb{e}$, i.e.,
\begin{equation} \label{eq:simplex:perm}
\hat{\cal U} = \beta\cdot{\sf conv}\left(\mb{e}_1,\mb{e}_2, \ldots,\mb{e}_m, \gamma \mb{e}\right)
\end{equation}
where $\beta$ is the scaling factor guaranteeing that $\hat{\cal U}$ dominates ${\cal U}$. This dominating set $\hat{\cal U}$ is motivated by the symmetry of the permutation invariant set $\cal U$. In this section, we show that one can efficiently compute the minimum $\beta$ such that $\hat{\cal U}$ in~\eqref{eq:simplex:perm} dominates ${\cal U}$. In particular, we derive an efficiently computable closed-form expression for $\beta$, for any PI set ${\cal U}$.

From Theorem \ref{thm:addi} we know that to construct a piecewise affine policy with an approximation bound of $2 \beta$, it is sufficient to find $\beta$ such that
\begin{equation}\label{eq:ineq:perm}
 \frac{1}{\beta} \max _{ \mb h \in {\cal U }}  \sum_{i=1}^m \left( h_i -\beta \gamma  \right)^{+} \leq 1
\end{equation}
\red{and any $\beta$ implies that $2\beta\cdot{\sf conv}\left(\mb{e}_1,\mb{e}_2, \ldots,\mb{e}_m, \gamma \mb{e}\right)$ dominates ${\cal U}$ (see Lemma \ref{lem:ineq}b)}. Finding the minimum $\beta$ that satisfies \eqref{eq:ineq:perm} requires solving:
\begin{equation}\label{eq:beta}
\min \left\{ \beta \geq 1 \; \Big\vert \;   \frac{1}{\beta} \max _{ \mb h \in {\cal U }}  \sum_{i=1}^m \left( h_i -\beta \gamma \right)^{+} \leq 1   \right\}.
\end{equation}
The following lemma characterizes the structure of the optimal solution for the maximization problem in~\eqref{eq:ineq:perm} for a fixed $\beta$. 
\begin{lemma} \label{lem:struct:sol}
Consider the maximization problem in \eqref{eq:ineq:perm} for a fixed $\beta$. There exists an optimal solution $\mb h^*$ such that
\[ \mb h^* = \gamma(k) \cdot \sum_{i=1}^k \mb e_i,\]
for some $k=1,\ldots,m$.
\end{lemma}
We present the proof of Lemma \ref{lem:struct:sol} in Appendix \ref{apx-proofs:lem:struct:sol}. The following lemma characterizes the optimal $\beta$ for~\eqref{eq:beta}.

\begin{lemma}\label{thm:beta}
\red{Let ${\cal U}$ be  a permutation invariant uncertainty set satisfying Assumption $1$}. Then the optimal solution for~\eqref{eq:beta} is given by
\begin{equation}\label{eq:sol:beta}
\beta  = \max _{ k=1,\ldots,m}  \frac{\gamma(k)}{\gamma +\frac{1}{k}}.
\end{equation}
\end{lemma}
\begin{proof}
Using Lemma~\ref{lem:struct:sol}, we can reformulate~\eqref{eq:beta} as follows.
\begin{align*}
  \min \left\{ \beta \geq 1 \; \Big\vert \;   \frac{1}{\beta} \max _{ k=1,\ldots,m    }  \sum_{i=1}^k \left( \gamma(k) -\beta \gamma \right) \leq 1   \right\},
\end{align*}
i.e.,
\begin{align*}
 \min \left\{ \beta \geq 1 \; \Big\vert \;   \beta \geq \frac{\gamma(k)}{\gamma +\frac{1}{k}}, \; \; \forall k=1,\ldots,m \right\}.
 \end{align*}
Therefore,
$$\beta  = \max _{ k=1,\ldots,m}  \frac{\gamma(k)}{\gamma  +\frac{1}{k}}.  $$
 \end{proof}

The above lemma computes the minimum $\beta$ that satisfies~\eqref{eq:ineq:perm}. Therefore, from Theorem~\ref{thm:addi}, we have the following theorem.
\begin{theorem}\label{thm:perm-inv-beta}
Let ${\cal U}$ be a permutation invariant set satisfying Assumption 1. Let $\gamma = \gamma(m)$ be as defined in~\eqref{eq:gamma(k)} and $\beta$ be as defined in~\eqref{eq:sol:beta}, and
\[ \hat{\cal U} = \beta \cdot {\sf conv } \left( \mb e_1,\ldots,\mb e_m, \gamma \mb e \right).\]
Let $\mb{\hat{x}}, \mb{\hat{y}(\hat{h})}$ for $\mb{\hat{h}} \in \hat{\cal U }$ be an optimal solution for $\Pi_{\sf AR}(\hat{\cal U})$~\eqref{eqn:ar}. Then the following piecewise affine solution
\begin{equation} \label{perm_policy}
\begin{split}
 \mb x  & =  2 \hat{\mb x}  \\
 \mb y( \mb h )  & = \frac{1}{\beta}\sum_{i=1}^m \left( h_i -\beta \gamma \right)^+ \hat{\mb y }(  \beta \mb e_i)+ \hat{\mb y }( \beta \gamma \mb e) \qquad  \qquad \forall \mb h \in \cal U,
\end{split}
\end{equation}
gives a $2 \beta$-approximation for $\Pi_{\sf AR}({\cal U})$ \eqref{eqn:ar}. Moreover, the set $2 \cdot \hat{\cal U}$ dominates ${\cal U}$.
\end{theorem}
The last claim that $2 \cdot \hat{\cal U}$ dominates ${\cal U}$ is a straightforward consequence of part$ (b)$ of Lemma \ref{lem:ineq}.

As a consequence of Theorem \ref{thm:perm-inv-beta}, for any permutation invariant uncertainty set, ${\cal U}$, we can compute the piecewise-affine policy for $\Pi_{\sf AR}({\cal U})$\eqref{eqn:ar} efficiently. In fact, for many cases, even more efficiently than an affine policy.

\subsection{Examples}

We present the approximation bounds for several permutation invariant uncertainty sets that are commonly used in the literature and in practice, including norm balls, intersection of norm balls and budget of uncertainty sets. In particular, it follows that for these sets, the performance bounds of our piecewise affine policy are significantly better than the best known performance bounds for affine policy.


\begin{prop} {\bf (Hypersphere)} \label{prop:sphere}
Consider the uncertainty set ${\cal U}=\{\mb{h}\in{\mathbb R}^{m}_+\;|\;||\mb{h}||_2 \leq 1\}$ which is the intersection of the unit hypersphere and the nonnegative orthant. Then,
\[\hat{\cal U}=m^{\frac{1}{4}} \cdot{\sf conv}\left(\mb{e}_1, \mb{e}_2,\ldots, \mb{e}_m, \frac{1}{\sqrt{m}}\mb{e} \right),\]
dominates ${\cal U}$ and our piecewise affine solution~\eqref{perm_policy} gives $O  ( m^{\frac{1}{4}} )$ approximation to $\eqref{eqn:ar}$.
\end{prop}
\begin{proof}
We have for $k=1,\ldots,m$,
$$ \gamma(k)   = \frac{1}{k} \cdot \max \left\{ \sum_{i=1}^k h_i \vert\ \mb h \in {\cal U} \right\} = \frac{1}{\sqrt{k}}.$$
In particular, $\gamma = \frac{1}{\sqrt{m}}$. From Lemma~\ref{thm:beta} we get,
\begin{align*}
\beta  &= \max _{ k=1,\ldots,m}  \frac{\gamma(k)}{\gamma (m) +\frac{1}{k}} \\
& = \max _{ k=1,\ldots,m}  \frac{\frac{1}{\sqrt{k}}}{\frac{1}{\sqrt{m}} +\frac{1}{k}}.
\end{align*}
The maximum of this problem occurs for $ k = \sqrt{m}$. Then, $ \beta = \frac{m^{\frac{1}{4}}}{2} $.  We conclude from Theorem~\ref{thm:perm-inv-beta} that $\hat{\cal U}$ dominates ${\cal U}$ and our piecewise affine policy gives $O  ( m^{\frac{1}{4}} )$ approximation to the adjustable problem $\eqref{eqn:ar}$.
 \end{proof}

\noindent
\red{
\textbf{Remark.}
Consider the following ellipsoid uncertainty set
\begin{equation} \label{set:ellipscale}
\left\{ \mb h \geq \mb 0 \; \left | \;  \sum_{i=1}^m r_i h_i^ 2 \leq 1 \right.\right\}.
\end{equation}
This is widely used to model uncertainty in practice and is just a diagonal scaling of the hypersphere uncertainty set. As we mention before, the performance of our policy is not affected by scaling.
 Hence, our piecewise affine policy gives an $O  ( m^{\frac{1}{4}} )$-approximation  to the adjustable problem $\eqref{eqn:ar}$ for ellipsoid uncertainty sets~\eqref{set:ellipscale} similar to hypersphere. We analyze the case of more general ellipsoids in Proposition \ref{prop:ellipsoid}.
 }

\begin{prop} {\bf{(p-norm ball)}}\label{prop:p-ball} Consider the p-norm ball uncertainty set $ {\cal U}= \left\{ \mb h \in \mathbb R^{m}_+ \;  \big\vert \; \| \mb h\| _p \leq 1 \right\}$ where $p \geq 1 $. Then
$$ \hat{ \cal U} = 2 \beta \cdot {\sf{conv}} \left( \mb{e_1}, \mb{e_2},\ldots, \mb{e_m},m^{-\frac{1}{p}} \mb e \right) $$
dominates ${\cal U}$ with
\[ \beta = \frac{1}{p}   (p-1)^{\frac{p-1}{p}}             \cdot m^{\frac{p-1}{p^2}} = O (  m^{\frac{p-1}{p^2}} ).\]
Our piecewise affine solution \eqref{perm_policy} gives $O  ( m^{\frac{p-1}{p^2}} )$ approximation to $\eqref{eqn:ar}$.
\end{prop}
\begin{proof}
We have for $k=1,\ldots,m$,
$$ \gamma(k)   = \frac{1}{k} \cdot \max \left\{ \sum_{i=1}^k h_i \vert\ \mb h \in {\cal U} \right\} = k^{\frac{-1}{p}}.$$
In particular, $\gamma =m^{\frac{-1}{p}}$. From Lemma~\ref{thm:beta} we get,
\begin{align*}
\beta  &= \max _{ k=1,\ldots,m}  \frac{\gamma(k)}{\gamma (m) +\frac{1}{k}} \\
& = \max _{ k=1,\ldots,m}  \frac{k^{\frac{-1}{p}}}{m^{\frac{-1}{p}} +\frac{1}{k}} \\
&= \frac{1}{p}   (p-1)^{\frac{p-1}{p}}             \cdot m^{\frac{p-1}{p^2}} = O \left(  m^{\frac{p-1}{p^2}} \right). \\
\end{align*}
We conclude from  Theorem~\ref{thm:perm-inv-beta} that $\hat{\cal U}$ dominates ${\cal U}$ and our piecewise affine policy gives $O (  m^{\frac{p-1}{p^2}} )$  approximation to the adjustable problem $\eqref{eqn:ar}$.
 \end{proof}

\begin{prop} {\bf{(Intersection of two norm balls)}}\label{prop:twoball} Consider $ {\cal U}$ the intersection of the norm balls  $ {\cal U}_1= \left\{ \mb h \in \mathbb R^{m}_+ \;  \big\vert \; \| \mb h\| _p \leq 1 \right\}$  and \\ $ {\cal U}_2= \left\{ \mb h \in \mathbb R^{m}_+ \;  \big\vert \; \| \mb h\| _q \leq r \right\}$ where $ p > q \geq 1 $ and $m^{\frac{1}{q}-\frac{1}{p}} \geq r \geq 1$. Then,
 $$ \hat{ \cal U} = \beta \cdot {\sf{conv}} \left( \mb{e_1}, \mb{e_2},\ldots, \mb{e_m}, \left( r m^{-\frac{1}{q}}\right) \mb e \right),$$
 where
 \[ \beta = \min( \beta_1, \beta_2), \; \beta_1= r^{\frac{1-p}{p}}m^{\frac{p-1}{pq}}, \mbox{ and } \beta_2=r^{\frac{1}{q}}m^{\frac{q-1}{q^2}}.\]
 Our piecewise affine solution \eqref{perm_policy} \red{gives a $2 \beta$ approximation }to $\eqref{eqn:ar}$.
\end{prop}
\begin{proof}
To prove that $\hat{ \cal U}$ dominates ${\cal U}_1 \cap { \cal U}_2$, it is sufficient to consider $\mb h$ in the boundary of ${\cal U}_1$ or ${ \cal U}_2$  and find $ \alpha_1, \alpha_2,\ldots,\alpha_{m+1} \geq 0$  with $\alpha_1+ \ldots+\alpha_{m+1} =1$ such that for all $ i \in [m],$
$$h_i \leq \beta \left( \alpha_i + rm^{-\frac{1}{q}}\alpha_{m+1}\right).$$

\vspace{2mm}
\noindent {\bf {Case 1:}} $\beta = \beta_1.$\\
Let $ \mb h \in {\cal U}_1$ such that $\| \mb h\|_p=1$, we take $\alpha_i = \frac{h_i^p}{p}$ for $ i \in [m]$ and $\alpha_{m+1}= \frac{p-1}{p}.$ First, we have $\sum_{i=1}^{m+1} \alpha_i =1$ and for all $ i \in [m]$,
\begin{align*}
 \beta \left( \alpha_i +rm^{-\frac{1}{q}}\alpha_{m+1}\right) &=  \beta_1 \left( \frac{h_i^p}{p}+ \frac{p-1}{p} rm^{-\frac{1}{q}} \right)\\
 & \geq \beta_1  \left( h_i^p \right)^{\frac{1}{p}} \left(r m^{-\frac{1}{q}}\right)^{\frac{p-1}{p}} = h_i,
\end{align*}
\red{where the inequality follows from the weighted inequality of arithmetic and geometric means (known as Weighted AM-GM inequality)}. Therefore $\hat{ \cal U}$ dominates ${\cal U}_1 \cap { \cal U}_2 $. \\

\vspace{2mm}
\noindent{\bf{Case 2:}} $\beta = \beta_2.$\\
Let $ \mb h \in {\cal U}_2$ such that $\| \mb h\|_q=r$, we take $\alpha_i = \frac{h_i^q}{r^q q}$ for $ i \in [m]$ and $\alpha_{m+1}= \frac{q-1}{q}.$ First, we have $\sum_{i=1}^{m+1} \alpha_i =1$ and for all $ i \in [m]$,
\begin{align*}
 \beta \left( \alpha_i +rm^{-\frac{1}{q}}\alpha_{m+1}\right) &=  \beta_2 \left( \frac{h_i^q}{r^q q}+ \frac{q-1}{q} rm^{-\frac{1}{q}} \right)\\
 & \geq \beta_2  \left( \frac{h_i^q}{r^q} \right)^{\frac{1}{q}} \left(r m^{-\frac{1}{q}}\right)^{\frac{q-1}{q}} = h_i,
\end{align*}
\red{where the inequality followed from the weighted AM-GM inequality}. Therefore, $\hat{ \cal U}$ dominates ${\cal U}_1\cap{ \cal U}_2 $.
 \end{proof}

We also consider a permutation invariant uncertainty set that is the intersection of  an ellipsoid and the non-negative orthant , i.e.,
\begin{equation} \label{ellip}
{\cal U}=\left\{\mb{h}\in{\mathbb R}^{m}_+\;|\;   \mb{h}^T \mb \Sigma  \mb{h}       \leq 1\right\}
\end{equation}
where $\mb \Sigma \succeq \mb 0$. For ${\cal U}$ to be a permutation invariant set satisfying Assumption 1, $\mb {\Sigma}$ must be of the following form
\begin{equation} \label{matrix:ellip}
 \mb \Sigma= \left(
\begin{matrix}
1 &  a & \ldots & a\\
a &  1 & \ldots & a\\
\vdots & \vdots & \ddots & \vdots\\
a  &   a       &\ldots & 1
\end{matrix}
\right)
\end{equation}
where $ 0\leq a \leq 1$.

\begin{prop} {\bf (Permutation invariant ellipsoid)} \label{prop:ellipsoid}
Consider the uncertainty set  ${\cal U}$  defined in \eqref{ellip} where  $\mb \Sigma$ is defined in \eqref{matrix:ellip}.
 Then
 \[\hat{\cal U}=\beta\cdot{\sf conv}\left(\mb{e}_1, \mb{e}_2,\ldots, \mb{e}_m, \gamma\mb{e} \right),\]
dominates ${\cal U}$ with $$\beta=   \left(  \frac{a}{2}+ \frac{(1-a)^{\frac{1}{2}}}{\left( am^2+(1-a)m\right)^{\frac{1}{4}}}\right)^{-1}= O \left(  m^{\frac{2}{5}} \right)  $$
 and $$ \gamma = \frac{1}{ \sqrt{\left( am^2+(1-a)m \right)}  } .   $$
 Our piecewise affine policy  \eqref{perm_policy} gives $ O \left(  m^{\frac{2}{5}} \right)  $ approximation to the adjustable robust problem $\eqref{eqn:ar}$.
\end{prop}
The proof of Proposition~\ref{prop:ellipsoid} is presented in Appendix \ref{apx-proofs:prop:ellipsoid}.

\begin{prop} {\bf{(Budget of uncertainty set)}} \label{prop:Budget of uncertainty set} Consider the budget of uncertainty set
\begin{equation}\label{set:budget of uncertainty}
\red{ {\cal U}= \left\{ \mb h \in [0,1]^{m} \;  \big\vert \; \sum_{i=1}^m h_i \leq k \right\}.}
\end{equation}
Then, $$ \hat{ \cal U} = \beta \cdot {\sf{conv}} \left( \mb{e_1}, \mb{e_2},\ldots, \mb{e_m}, \frac{k}{m} \mb e \right)$$ where $ \beta = \min \left( k, \frac{m}{k}\right) $.
 In particular, our piecewise affine policy \eqref{perm_policy} gives $ 2 \beta $ approximation to the adjustable problem $\eqref{eqn:ar}$.
\end{prop}
The proof of Proposition \ref{prop:Budget of uncertainty set} is presented in Appendix \ref{apx-proofs:prop:Budget of uncertainty set}.
\red{
\subsection{Comparison to affine policy}}
Table \ref{tab:results} summarizes the performance bounds for our piecewise affine policy and the best known performance bounds in the literature for affine policies~\cite{BB15}.  As can be seen, \red{our piecewise affine policy performs significantly better than the known bounds for affine policy for many interesting sets, including hypersphere, ellipsoid and norm-balls. For instance, our policy gives $ O ( m^{\frac{1}{4}})$-approximation for the hypersphere and   $O(m^{\frac{p-1}{p^2}})$-approximation for the $p$-norm ball, while affine policy gives $O(m^{\frac{1}{2}})$-approximation for hypersphere and $ O(m^{\frac{1}{p}})$-approximation for the $p$-norm ball~\cite{BB15}, respectively. However, as we mentioned before, our policy is not a generalization of affine policies and, in fact, affine policies may perform better for certain uncertainty sets.  However, we present a family of examples where an optimal affine policy gives an $\Omega (\sqrt{m}) $-approximation, while our policy is \emph{near-optimal} for the adjustable robust problem~\eqref{eqn:ar}. In particular, we consider the following instance motivated from the worst-case examples of affine policy in~\cite{BG10} and \cite{housni2017beyond}.
\begin{equation} \label{ex:bad}
\begin{aligned}
&n=m, \; \;   r= \lceil m- \sqrt{m} \rceil ,\; \;  N= \binom m r \\
&     {B}_{ij}= \left\{
    \begin{array}{ll}
        1 & \mbox{if} \; \; i=j \\
          \frac{1}{\sqrt{m}} 		& \mbox{if}  \; \;   i \neq j \\
    \end{array}
\right.\\
& \mb A= \mb B, \; \;  \mb c=\frac{1}{15}\mb e, \; \;  \mb d = \mb e \\
&{\cal U} = {\sf{conv}} \left( \mb 0, \mb e_1 , \ldots ,\mb e_m, \mb \nu_1, \ldots, \mb \nu_N \right) \\
&\text{        where   } \mb \nu_1 =  \frac{1}{\sqrt{m}} \cdot [ \underbrace{1,\ldots,1}_r ,0\ldots,0 ] ;
 \end{aligned}
\end{equation}
$ \mb \nu_1$ has exactly $r$  non-zero coordinates, each equal to $\frac{1}{\sqrt{m}}$. The extreme points $\mb \nu_i$ of $\nu_1$, are permutations of the non-zero coordinates of $\mb \nu_{1}$.  Therefore, ${\cal U}$ has exactly $ \binom m r  +m+1$ extreme points.}

\red{\begin{lemma}\label{lem:tight-comp-paff}
Our piecewise affine policy \eqref{picewise_sl} gives an $ O (1+ \frac{1}{\sqrt{m}})$-approximation for the adjustable robust problem~\eqref{eqn:ar} for instance \eqref{ex:bad}. 
\end{lemma}
We can prove Lemma \ref{lem:tight-comp-paff} by constructing a dominating set within a scaling factor $ O (1+ \frac{1}{\sqrt{m}})$ from ${\cal U}$. We present the complete proof of Lemma \ref{lem:tight-comp-paff} in Appendix \ref{apx-proofs:lem:tight-comp-paff}.
\begin{lemma}\label{lem:tight-comp-aff}
Affine policy  gives an $ \Omega ({\sqrt{m}})$-approximation for the adjustable robust problem~\eqref{eqn:ar} for instance \eqref{ex:bad}. Moreover, for any optimal affine solution, the cost of the first-stage solution $ \mb x^*_{\sf Aff}$ is $ \Omega ({\sqrt{m}})$ away from the optimal adjustable problem ~\eqref{eqn:ar}, i.e. $ \mb c^T \mb x^*_{\sf Aff} = \Omega(m^{1/2}) \cdot z_{\sf AR}({\cal U})$.
\end{lemma}
We present the proof of Lemma \ref{lem:tight-comp-aff} in Appendix \ref{apx-proofs:lem:tight-comp-aff}. From Lemma \ref{lem:tight-comp-aff} and \ref{lem:tight-comp-paff}, we conclude that our policy is near-optimal whereas affine policy is $ \Omega ( \sqrt{m})$ away from the optimal adjustable solution for the instance \eqref{ex:bad}. Hence our policy provides a significant improvement.  We would like to note that since $\hat{\cal U}$ is a simplex, an affine policy is optimal for $\Pi_{\sf AR}(\hat{\cal U})$. In particular, we have the following
\[z_{\sf AR}({\cal U}) \leq z_{\sf AR}(\hat{\cal U}) = z_{\sf Aff}(\hat{\cal U}) \leq O \left(1+\frac{1}{\sqrt{m}}\right) \cdot z_{\sf AR}({\cal U}),\]
where the first inequality follows as $\hat{\cal U}$ dominates ${\cal U}$ and the last inequality follows from Lemma~\ref{lem:tight-comp-paff}. Moreover, from Lemma~\ref{lem:tight-comp-aff}, we know that for instance~\eqref{ex:bad},
\[ z_{\sf Aff}({\cal U}) = \Omega(\sqrt{m}) \cdot z_{\sf AR}({\cal U}).\]
Therefore,
\[z_{\sf Aff}({\cal U}) = \Omega(\sqrt{m}) \cdot z_{\sf Aff}(\hat{\cal U}),\]
which is quite surprising since $\hat{\cal U}$ dominates ${\cal U}$. We would like to emphasize that $\hat{\cal U}$ only dominates ${\cal U}$ and does not contain it and this is crucial to get a significant improvement for our piecewise affine policy constructed through the dominating set.
\vspace{2mm}
\\
\noindent
{\bf Comparison to re-solving policy:}
In many applications, a practical implementation of affine policy only implements the first stage solution $\mb x^*_{\sf Aff}$ and re-solve (or recompute) the second-stage solution once the uncertainty is realized. The performance of such a re-solving policy is at least as good as affine policy and in many cases significantly better. Lemma \ref{lem:tight-comp-aff} shows that for instance \eqref{ex:bad}, such a re-solving policy is $\Omega(\sqrt{m})$ away from the optimal adjustable policy whereas we show in  Lemma \ref{lem:tight-comp-paff}
that our piecewise affine policy is near-optimal. Hence, our piecewise affine policy for instance \eqref{ex:bad} is  performing significantly better not only than affine policy but also the re-solving policy.
}

\section{General uncertainty set} \label{sec:gen}

In this section, we consider the case of general uncertainty sets. The main challenge in our framework of constructing the piecewise affine policy is the choice of the dominating simplex, $\hat{\cal U }$. More specifically, the choice of $\beta$ and $\mb v \in {\cal U}$ such that $ \beta \cdot {\sf conv } \left( \mb e_1,\ldots,\mb e_m, \mb v \right)$ dominates ${\cal U }$. For a permutation invariant set, ${\cal U}$, we choose $\mb v= \gamma  \mb e$ and we can efficiently find $ \beta$ using Lemma~\ref{thm:beta} to construct the dominating set. However, this does not extend to general sets and we need a new procedure to find those parameters.

Theorem \ref{thm:addi} shows that to construct a good piecewise affine policy over ${\cal U}$, it is sufficient to find $\beta$ and $\mb v \in {\cal U}$  such that for all $\mb h \in {\cal U}$
\begin{equation} \label{eq:gn:ine}
  \frac{1}{\beta} \sum_{i=1}^m \left( h_i -\beta v_i  \right)^{+} \leq 1 .
\end{equation}
In this section, we present an iterative algorithm to find such $\beta$ and $\mb v \in {\cal U}$ satisfying~\eqref{eq:gn:ine}. In each iteration $t$, the algorithm maintains a candidate solution, $\beta^t$ and $\mb v^t \in {\cal U}$. Let $\mb u^t = \beta^t \cdot v^t$.  The algorithm solves the following maximization problem:
\begin{equation}\label{eq:algo2-iter}
\max_{\mb h \in {\cal U}} \sum_{i=1}^m \left( h_i - u^t_i  \right)^{+}
\end{equation}
The algorithm stops if the optimal value is at most $\beta^t$ in which case, Condition \eqref{eq:gn:ine} is verified for all $h \in {\cal U}$. Otherwise, let $\mb h^t$ be an optimal solution of problem ~\eqref{eq:algo2-iter}. The current solutions are updated as follows:
\[
\begin{aligned}
\beta^{t+1} & = \beta^t + 1 \\
u^{t+1}_i & = \min \left( 1, u^t_i + h^t_i\right).
\end{aligned}
\]
This corresponds to updating $\mb v^{t+1} = \frac{1}{\beta^{t+1}} \cdot \mb u^{t+1}$. Algorithm~\ref{algo2} presents the steps in detail.

\begin{algorithm}[t]
\caption{Computing $\beta$  and $\mb v$ for general uncertainty sets}\label{algo2}
\begin{algorithmic}[1]
\State Initialize $ t =0, \;  \mb u^0 =0$
\While{ $  \left\{  \underset{\mb h \in {\cal U}} \max \sum_{i=1}^m \left( h_i-u_i^t\right)^+ > t  \right\} $}
\State $ \mb h^t \in  \underset{\mb h \in {\cal U}}{  \sf argmax} \sum_{i=1}^m \left( h_i-u_i^t \right)^+$
 \For{ $i=1,\ldots,m$}
  \If {$ u_i^t= 1$} $ h_i^t= 0$ \EndIf
 \State $ u_i^{t+1}= \min ( 1, u_i^t +h_i^t)$
 \EndFor
\State $t  =t+1 $
\EndWhile
\State \Return $ \beta =t, \; \mb v= \frac{\mb u^t}{\beta}  . $
\end{algorithmic}
\end{algorithm}


The number of $ \beta$-iterations  is finite since ${\cal U}$ is compact.  The following theorem shows that $ \mb v $ returned by the algorithm belongs to ${\cal U}$ and the corresponding piecewise affine policy is a $O(\sqrt m)$-approximation for the adjustable  problem~\eqref{eqn:ar}.

\begin{theorem} \label{thm:bound:beta}
Suppose Algorithm~\ref{algo2} returns $\beta$, $\mb v$. Then $\mb v \in {\cal U}$. Furthermore, the piecewise affine policy~\eqref{picewise_sl} with parameters $\beta$ and $\mb v$ gives a $O(\sqrt m)$-approximation for the adjustable problem~\eqref{eqn:ar}.
\end{theorem}
\begin{proof}
Suppose Algorithm~\ref{algo2} returns $\beta, \mb v$. Note that $\beta$ is the number of iterations in Algorithm \ref{algo2}. First, we have
$$ \mb u^{\beta} \leq \sum_{t=0}^{\beta -1} \mb h^t .$$
Moreover $\frac{1}{\beta}\cdot \sum_{t=0}^{\beta -1} \mb h^t \in {\cal U}$ since ${\cal U}$ is convex. Therefore $ \mb v = \frac{\mb u^{\beta}}{\beta} \in {\cal U}$ by down-monotonicity of ${\cal U}$.

\red{Let us prove that $\beta =O ( \sqrt{m} )$. First, note that, when we set $h_i^t=0$ for $u_i^t=1$, the objective of the maximization problem in the algorithm does not change and $\mb h^t $ still belongs to $ {\cal U}$ by down-monotonicity}. Then, for any $ t=0, \ldots , \beta-1$
$$   \sum_{i=1}^m \left( h_{i}^t- u_i^t \right)^+ > t  .$$
Moreover, $h_i^t \geq 0 $ and $u_i^t \geq 0 $, hence $h_i^t \geq (h_i^t - u_i^t )^+$ and therefore for all  $t=	0,\ldots,\beta-1$
$$\sum_{i=1}^m  h_i^t >  t.$$
Then,
\begin{equation} \label{eq:tech:ineq}
\sum_{t=0}^{\beta-1} \sum_{i=1}^m  h_i^t >   \sum_{t=0}^{\beta-1}  t  =    \frac{1}{2}    \beta ( \beta -1 ) .
\end{equation}
Note that, if $u_i^t=1$ at some iteration $t$, then $h_i^{t'}=0$ for any $t' \geq t$. Hence, for any $i\in [m]$,
\begin{equation} \label{eq:tech:ineq2}
 \sum_{t=0}^{\beta-1}  h_i^t \leq   u_i^{\beta} +1 \leq 2 .
\end{equation}
Hence, from \eqref{eq:tech:ineq} and  from \eqref{eq:tech:ineq2} we get,
$ 2m >   \frac{1}{2}    \beta ( \beta -1 )$, i.e.,  $ \beta \cdot (\beta -1) \leq 4m ,$
which implies, $\beta = O ( \sqrt{m} ).$
 \end{proof}

We note that the maximization problem~\eqref{eq:algo2-iter} that Algorithm~\ref{algo2} solves in each iteration $t$ is not a convex optimization problem.  However, ~\eqref{eq:algo2-iter} can be formulated as the following MIP:
\begin{equation}\label{eqn:IP}
\begin{aligned}
\max \; & \sum_{i=1}^m z_i \\
&   z_i \leq (h_i-u_i^t) +(1-x_i) \; \; \forall i \in [m],     \\
& z_i \leq x_i \; \; \forall i \in [m]        \\
&  z_i \geq 0, \; \; \forall i \in [m]       \\
&  x_i \in \{0,1\} \; \; \forall i \in [m] \\
&\mb h \in{\cal U} .
\end{aligned}
\end{equation}
Therefore, for general uncertainty set ${\cal U}$, the procedure to find $\beta$ and $\mb v \in {\cal U}$ is computationally more challenging than for the case of permutation invariant sets.

\textbf{Remark.}
\red{Since the computation of $\beta$ and $\mb v$ depends only on ${\cal U}$, and not on the problem parameters (i.e., the parameters $\mb A, \mb B, \mb c$ and $\mb d$)}, one can compute them offline and then use them to efficiently construct a good piecewise affine policy.

\vspace{3mm}
\noindent {\bf Connection to Bertsimas and Goyal~\cite{BG10}}. We would like to note that Algorithm~\ref{algo2} is quite analogous to the explicit construction of good affine policies in~\cite{BG10}. The analysis of the $O(\sqrt m)$-approximation bound for affine policies is based on the following projection result (which is a restatement of Lemma 8 and Lemma 9  in~\cite{BG10}).

\begin{theorem}{\bf[Bertsimas and Goyal 2011]}\label{thm:berti-goyal}
Consider any uncertainty set ${\cal U}$ satisfying Assumption $1$. There exists  $\beta \leq \sqrt m$, $\mb v \in {\cal U}$ such that
\[ \sum_{ j : \beta v_j < 1} \; h_j \leq \beta, \; \forall \mb h \in {\cal U}.\]
\end{theorem}
Suppose $J = \{ j \; |\; \beta v_j < 1\}$. The affine solution in~\cite{BG10} covers $\beta \mb{v}$ using the static component and the components $J$ using a linear solution.  The linear solution does not exploit the coverage of $\beta v_i$ for $ i \in J$ from the static solution. The approximation factor is $O(\beta)$ since for all $\mb h \in {\cal U}$,  $\sum_{j \in J} h_j \leq \beta$.

Our piecewise affine solution given by Algorithm~\ref{algo2} finds analogous $ \beta$, $\mb v \in {\cal U}$ such that
\[ \sum_{i=1}^m ( h_i  - \beta v_i) _+ \leq \beta, \; \forall \mb h \in {\cal U}.\]
\red{In the piecewise affine solution, the static component covers  $\beta \mb v$ and the remaining part $( \mb h  -  \beta \mb v)_+$ is covered by a piecewise-linear function that exploits the coverage of $\beta \mb v$}. This allows us to improve significantly as compared to the affine policy for a large family of uncertainty sets. We would like to note again that our policy is not necessarily an optimal one and there can be examples where affine policy is better than our policy.

\section{A worst case example for the domination policy}\label{sec:worst-case}

From Theorem \ref{thm:bound:beta}, we know that our piecewise affine policy gives an $O(\sqrt{m})$-approximation for the adjustable robust problem~\eqref{eqn:ar}. In this section, we show that this bound is tight for the following budget of uncertainty set:
\begin{equation}\label{set-ones}
 {\cal U}= \left\{ \mb h \in \mathbb R^{m}_+ \;  \bigg\vert \;     \sum_{i=1}^m h_i= \sqrt{m} , \; 0 \leq h_i \leq 1 \ \;  \forall i \in [m]       \right\}.
\end{equation}
 We show that our dominating simplex based piecewise affine policy gives an $\Omega (\sqrt{m})$-approximation to the adjustable robust problem~\eqref{eqn:ar}. The lower bound of $\Omega (\sqrt{m})$ holds even when we consider more general dominating sets than simplex. We show that for any $ \epsilon > 0 $, there is no polynomial number of points in ${\cal U}$ such that the convex hull of those points scaled by $ m ^{\frac{1}{2}-\epsilon}$  dominates $ {\cal U}$. In particular, we have the following theorem.

\begin{theorem}\label{thm:worstcase}
Given any $0 < \epsilon < 1/2$, and $k \in {\mathbb N}$, consider the budget of uncertainty set, ${\cal U}$~\eqref{set-ones} with $m$ sufficiently large. Let $P(m) \leq m^k$. Then for any $\mb{z}_1,\mb{z}_2,\ldots\mb{z}_{P(m)} \in {\cal U}$,  the set
\[ {\hat{\cal U}}= m^{\frac{1}{2}-\epsilon} \cdot { \sf{conv}} \left( \mb{z}_1,\mb {z}_2,\ldots\mb{z}_{P(m))}\right),\]
does not dominate  ${\cal U}$.
\end{theorem}

\begin{proof}
Suppose for a sake of contradiction that there exists $\mb{z}_1,\mb{z}_2,\ldots,\mb{z}_{P(m)} \in {\cal U}$ such that ${\hat{\cal U}}= m^{\frac{1}{2}-\epsilon} \cdot { \sf{conv}} \left( \mb{z}_1,\mb {z}_2,\ldots\mb{z}_{P(m)}\right)$ dominates  ${\cal U}$.

By Caratheodory's theorem, we know that any point in ${\cal U}$ can be expressed as a convex combination of at most $m+1$ extreme points of ${\cal U}$. Therefore
$${\hat{\cal U}} \subseteq m^{\frac{1}{2}-\epsilon} \cdot { \sf{conv}} \left( \mb {y}_1,\mb {y}_2,\ldots, \mb{y}_{Q(m))}\right), $$
where $\mb{y}_1,\mb{y}_2,\ldots,\mb{y}_{Q(m)} $ are extreme points of ${\cal U}$ and $$Q(m) \leq  (m+1) \cdot P(m) = O(m^{k+1}).$$

\red{Consider any $ I \subseteq \{1,2,\ldots,m\}$ such that $ \vert I \vert = \sqrt{m}$. Let $ {\mb h}$ be an extreme point of ${\cal U}$ corresponding to $I$, i.e., $ h_i= 1 $ if $ i \in I$ and $  h_i= 0 $ otherwise. Since we assume that $\hat{\cal U}$ dominates $\cal U$,  there exists $ {\hat{\mb h}} \in {\hat{\cal U}}$ such that $ \mb h \leq \hat{\mb h}$}. Let
$$ {\hat{\mb h}}=m^{\frac{1}{2}-\epsilon}  \sum_{j=1}^{Q(m)} { \alpha}_j \mb{y}_j,$$ where $ \sum_{j=1}^{Q(m)} { \alpha}_j=1$ and ${ \alpha}_j \geq 0 $ for all $j=1,2,\ldots,Q(m)$. \red{We have
$$ 1= h_i \leq \hat{h}_i \; \; \forall i \in I$$
i.e.
$$ 1 \leq m^{\frac{1}{2}-\epsilon}  \sum_{j=1}^{Q(m)} { \alpha}_j {y_{ji}}, \; \forall i \in I.$$ }
Summing over $i \in I$, we have,
\[\sqrt{m}= \vert I \vert \leq m^{\frac{1}{2}-\epsilon}  \sum_{i \in I} \sum_{j=1}^{Q(m)} { \alpha}_j {y_{ji}}.\]
Therefore,
\[
\begin{aligned}
 m^{\epsilon} & \leq   \sum_{j=1}^{Q(m)} { \alpha}_j \sum_{i \in I}  {y_{ji}}, \\
 &  \leq \left(  \sum_{j=1}^{Q(m)} { \alpha}_j \right) \cdot \underset{ j=1,2,\ldots,Q(m)} \max\sum_{i \in I}  {y_{ji}}\\
 &=  \underset{ j=1,2,\ldots,Q(m)} \max\sum_{i \in I}  {y_{ji}} = \sum_{i \in I}  {y_{j^*i}},
\end{aligned}
\]
\red{where the second inequality follows from taking the max of the inner sum over indices $j$ and  $j^*$ is the index corresponding to the maximum sum.}

\red{ Therefore, for any $ I \subseteq \{1,2,\ldots,m\}$ with cardinality $ \vert I \vert = \sqrt{m}$, there exists $j=1,2,\ldots, Q(m)$ such that
$$ \sum_{i \in I}  {y_{ji}} \geq m^{\epsilon}.$$
}
\red{Denote ${\cal F} = \left\{  I \subseteq \{1,2,\ldots,m\} \; \big\vert \; \vert I \vert=\sqrt{m} \right\}$ which represents the set of all subsets of $\{1,2,\ldots,m\}$ with cardinality $\sqrt{m}$.} Note that the cardinality of $\cal F$ is
$$ \vert {\cal F} \vert = \binom{m}{\sqrt{m}} .$$
We know that for any $ I \in {\cal F} $ there exists $\mb y _j \in \{ \mb{y}_1,\mb{y}_2,\ldots\mb{y}_{Q(m)} \} $ such that
\[ \sum_{i \in I}  {y_{ji}} \geq m^{\epsilon}.\]
\red{ We have $\binom{m}{\sqrt{m}}$ possibilities for $I$ and $Q(m)$ possibilities for $\mb y_j$, hence by the pigeonhole principle, there exists a fixed $\mb y \in \{ \mb{y}_1,\mb{y}_2,\ldots\mb{y}_{Q(m)} \} $ and $ \tilde{\cal F} \subseteq {\cal F} $ such that  }
\begin{equation}\label{cond}
\begin{aligned}
& \vert \tilde{\cal F} \vert  \geq \frac{1}{Q(m)} \binom{m}{\sqrt{m}}, \mbox{ and } \\
& \\
& \sum_{i \in I}  {y_{i}}  \geq m^{\epsilon}, \; \forall I \in \tilde{\cal F}.
\end{aligned}
\end{equation}
Note that $\mb y$ is an extreme point of $ \cal U$. Hence, $\mb y$ has exactly $ \sqrt{m}$ ones and the remaining components are zeros. \red{The maximum cardinality of subsets $I \subseteq [m]$ that can be constructed to satisfy $\sum_{i \in I}  {y_{i}}  \geq m^{\epsilon}$ is}
$$ \sum_{k=m^{\epsilon}}^{k=\sqrt{m}} \binom{\sqrt{m}}{k} \cdot \binom{m-\sqrt{m}}{\sqrt{m}-k} .$$
By over counting, the above sum can be upper-bounded by
$$ \binom{\sqrt{m}}{m^{\epsilon}} \cdot \binom{m-m^{\epsilon}}{\sqrt{m}-m^{\epsilon}}.$$
\red{Therefore, cardinality of $\tilde{\cal F}$ should be less that the above upper bound, i.e.}
$$\binom{\sqrt{m}}{m^{\epsilon}} \cdot \binom{m-m^{\epsilon}}{\sqrt{m}-m^{\epsilon}} \geq  \vert \tilde{\cal F} \vert \geq \frac{1}{Q(m)} \binom{m}{\sqrt{m}} $$
Then,
\begin{equation}\label{contr}
\frac{\binom{\sqrt{m}}{m^{\epsilon}} \cdot \binom{m-m^{\epsilon}}{\sqrt{m}-m^{\epsilon}} }{\binom{m}{\sqrt{m}}} \geq \frac{1}{Q(m)}.
\end{equation}
 which is a contradiction. \red{The contradiction is derived by analyzing the order of the fractions in \eqref{contr})} (see Appendix \ref{apx-proofs:worstcase}).  \end {proof}

\red{
\section{Computational study} \label{sec:sim}
  In this section, we present a computational study to compare the performance of our policy with affine policies both in terms of objective function value of problem $\Pi_{\sf AR}({\cal U})$
\eqref{eqn:ar}
  and computation times. We explore both cases of permutation invariant sets and non-permutations invariant sets.
 \subsection{Experimental setup}
  \noindent {\bf Uncertainty sets.}
 We consider the following classes of uncertainty sets for our computational experiments.
 \begin{enumerate}
 \item {\bf Hypersphere}. We consider the following unit hypersphere defined in~\eqref{def:sphere},
\begin{equation*}
{\cal U}=\{\mb{h}\in{\mathbb R}^m_+\;|\;||\mb{h}||_2\leq 1\}.
\end{equation*}
 \item {\bf p-norm balls}. We consider the following  sets defined in Proposition~\ref{prop:p-ball}.
 \begin{equation*}
{\cal U}= \left\{ \mb h \in \mathbb R^{m}_+ \;  \big\vert \; \| \mb h\| _p \leq 1 \right\}.
\end{equation*}
For our numerical experiments, we consider the cases of $p=3$ and $p=3/2$.
\item  {\bf Budget of uncertainty set}. We consider the following set defined in \eqref{set:budget of uncertainty},
\begin{equation*}
 {\cal U}= \left\{ \mb h \in [0,1]^{m} \;  \Bigg\vert \; \sum_{i=1}^m h_i \leq k \right\}.
\end{equation*}
Here, $k$ denotes the budget. For our numerical experiments, we choose $k=c  \sqrt{m}$  where $c$ is a random uniform constant between $1$ and $2$.
\item  {\bf Intersection of budget of uncertainty sets.} We consider the following intersection of $L$ budget of uncertainty sets:
\begin{equation} \label{set:inters:budget}
{\cal U}= \left\{ \mb h \in [0,1]^m \; \Bigg\vert \; \sum_{j=1}^m  \alpha_{ij} h_j \leq 1,  \; \forall  i=1,\ldots,L \right\}.
\end{equation}
Here, $\alpha_{ij}$ are non-negative scalars. Note that the intersection of budget of uncertainty sets are not permutation invariant. For our numerical experiments, we generate $\alpha_{ij} $ i.i.d. according to absolute value of standard Gaussians and we normalize $\vert\vert \mb \alpha_i \vert\vert_2$ to 1 for all $i$
(i.e. $ \mb\alpha_{i} = \vert \mb G_i \vert  / \vert \vert \mb G_i \vert\vert_2 $ where $\mb G_i $ are i.i.d. according to ${\cal N}(0, \mb I_m )$) and we consider $L=2$ and $L=5$.
\item {\bf Generalized budget of uncertainty set.} We consider the following set
\begin{equation} \label{set:polytope}
{\cal U}= \left\{ \mb h \in [0,1]^m \; \Bigg\vert \; \sum_{\ell=1}^m h_{\ell} \leq 1 + \theta ( h_i +h_j) \; \; \forall i \neq j \right\}.
\end{equation}
This is a generalized version of the budget of uncertainty set \eqref{set:budget of uncertainty} where the budget $\theta$ is not a constant but depends on the uncertain parameter $\mb h$. In particular, the budget in the set \eqref{set:polytope} depends on the sum of the two lowest components of $\mb h$. For our numerical experiments, we choose $\theta = O(m)$.
 \end{enumerate}
}

\vspace{2mm}
\noindent\\
\red{
{\bf Instances}. We construct test instances of the adjustable robust problem \eqref{eqn:ar} as follows. We choose $n=m$, $\mb c= \mb d = \mb e $ and $ \mb A= \mb B$ where $\mb B$ is  randomly generated as
\[ \mb B = \mb I_m + \mb G,\]
where $\mb I_m$ is the identity matrix and $\mb G$ is a random normalized gaussian. In particular, for the hypersphere  uncertainty set, the budget of uncertainty set,
 the intersection of budget of uncertainty sets and the generalized budget, we conisder $G_{ij} = |Y_{ij}|/ \sqrt{m}$. For the $3$-norm ball, $G_{ij} = |Y_{ij}|/ m^{\frac{1}{3}}$ and for the $\frac{3}{2}$-norm ball, $G_{ij} = |Y_{ij}|/ m^{\frac{2}{3}},$  where $Y_{ij}$ are i.i.d. standard gaussian. We consider values of $m$ from $m=10$ to $ m=100$ in increments of $10$ and consider $50$ instances for each value of $m$.
}

\vspace{2mm}
\noindent
\red{
{\bf Our piecewise affine policy.} We construct the piecewise affine policy based on the dominating simplex $\hat{\cal U}$ as follows. For permutation invariant sets, we use the dominating simplex that can be computed in closed form. In particular, for the hypersphere uncertainty set, we use the dominating set $\hat{\cal U}$  in Proposition~\ref{prop:sphere}. For the p-norm balls, we use the dominating set $\hat{\cal U}$ in Proposition~\ref{prop:p-ball}. For the  budget of uncertainty set, we use the dominating set $\hat{\cal U}$ in Proposition~\ref{prop:Budget of uncertainty set} and for the  generalized budget of uncertainty set \eqref{set:polytope}, we use the dominating set ${\hat{\cal U}}$ in Proposition \ref{prop:poly} (see Appendix \ref{apx-prop:poly}).}

\red{
For non-permutation invariant sets, we use Algorithm \ref{algo2} to compute the dominating simplex. In particular, we get $\beta$ and $\mb v$ that satisfies \eqref{ineq:beta-v}
and  $2 \beta \cdot {\sf conv } \left( \mb e_1,\ldots,\mb e_m, \mb v \right)$ is a dominating set (see Lemma \ref{lem:ineq}-b). We can also show that the following set \eqref{newdom} is a dominating set (see Proposition \ref{prop:inters} in Appendix \ref{apx-:prop:inters}),
\begin{equation} \label{newdom}
{\hat{\cal U}} = \beta \cdot {\sf conv} \left( \mb v, \mb e_1 + \mb v, \ldots, \mb e_m + \mb v \right).
\end{equation}
While the worst case scaling factor for the above dominating set can be $2 \beta$ and therefore the theoretical bounds do not change, computationally \eqref{newdom} can provide a better policy and we use this in our numerical experiments for the intersection of budget of uncertainty sets  \eqref{set:inters:budget}.
}

\red{
\subsection{Results}
Let $z_{\sf p-aff}({\cal U})$ denote the worst-case objective value of our piecewise affine police. Note that  the piecewise affine policy over ${\cal U}$ is computed by solving the adjustable robust problem over $\hat{\cal U}$ and $z_{\sf p-aff}({\cal U}) = z_{\sf AR}(\hat{\cal U})$. For each uncertainty set we report the ratio $r =  \frac{z_{\sf{Aff}} ({\cal U})    }{z_{\sf p-aff}  ({\cal U}) }$ for $m=10$ to $100$. In particular, for each value of $m$, we report the average ratio (${\sf Avg}$), the maximum ratio  (${\sf Max}$), the minimum ratio (${\sf Min}$), the quantiles $5\%, 10\% , 25\%, 50\% $ for the ratio $r$, the running time of our policy ($T_{\sf p-aff}(s)$) and the running time of affine policy ($T_{\sf aff}(s)$). In addition, for the intersection of budget of uncertainty sets, we also report the computation time to construct $\hat{\cal U}$ via Algorithm \ref{algo2} ($T_{\sf Alg1}(s)$).  The numerical results are obtained using Gurobi 7.0.2 on a 16-core server with 2.93GHz processor and 56GB RAM.}



 \vspace{2mm}
 \noindent
 \red{
{\bf Hypersphere and Norm-balls.}
We present the results of our computational experiments in Tables \ref{tab:simu:sphere}, \ref{tab:simu:3ball} and \ref{tab:simu:small:ball} for the hypersphere and norm-ball uncertainty sets. We observe that the piecewise affine policy performs significantly better than affine policy for our family of test instances. In Tables \ref{tab:simu:sphere}, \ref{tab:simu:3ball} and \ref{tab:simu:small:ball}, we observe that the ratio  $r= \frac{z_{\sf{Aff}} ({\cal U})    }{z_{\sf{p-aff}}  ({\cal U})         }$  increases significantly as $m$ increases which implies that our policy provides a significant improvement over affine policy for large values of $m$. We also observe that the ratio for the hypersphere is larger than the ratio for norm-balls. This matches the theoretical bounds presented in Table \ref{tab:results} which suggests that the improvement over affine policy is the highest for $p=2$ for p-norm balls. }

\red{
We note that for the smallest values of $m$ ($m =10$), the performance of affine policy is better than our policy.  However, for $m > 10$, the performance of our policy is significantly better for all these three uncertainty sets: hypersphere, 3-norm ball and 3/2-norm ball.}

\red{
Furthermore, our policy scales  well and the average running time is less than $0.1$ second even for large values of $m$. On the other hand, computing the optimal affine policy over ${\cal U}$ becomes computationally challenging as $m$ increases. For instance, the average running time for computing an optimal affine policy for $m=100$ is around $9$ minutes for the hypersphere uncertainty set, around $17$ minutes for the $3$-norm ball and around $16$ minutes for the $3/2$-norm ball. }

\vspace{2mm}
\noindent
\red{ {\bf Budget of uncertainty sets.}
We present the results of our computational experiments in Tables \ref{tab:simu:budget}, \ref{tab:simu:2bud}, \ref{tab:simu:5bud} and \ref{tab:simu:genpoly} for the single budget of uncertainty set, the intersection of budget sets and the generalized budget.}

\red{
For the budget of uncertainty set \eqref{set:budget of uncertainty}, we observe that affine policy performs better than our piecewise affine policy for our family of test instances. Note that as we mention earlier, our policy is not a generalization of affine policies and therefore is not always better. For our experiments, we use $k=c \sqrt{m}$ which gives the worst case theoretical bound for our policy (see Theorem \ref{thm:worstcase}), but the performance of our policy is still reasonable and the average ratio  $r= \frac{z_{\sf{Aff}} ({\cal U})    }{z_{\sf{p-aff}}  ({\cal U})         }$ over all instances is  around 0.88 as we can observe in Table \ref{tab:simu:budget}.  On the other hand, as in the case of conic uncertainty sets, our policy scales well with an average running time less than $0.1$ second even for large values of $m$, whereas affine policy takes for example more than 6 minutes on average for $m=100$. }

\red{
Tables \ref{tab:simu:2bud} and \ref{tab:simu:5bud} present the results for intersection of budget of uncertainty sets. We observe that affine policy outperforms our policy as in the case of a single budget. This confirms that affine policy performs very well empirically for this class of uncertainty sets. We also observe that the performance of our policy improves when we increase the number of budget constraints. For example, for $m=100$, the average ratio $r= \frac{z_{\sf{Aff}} ({\cal U})    }{z_{\sf{p-aff}}  ({\cal U})         }$ increases  from $0.79$ in the case of $L=2$ to $0.88$ for $L=5$. This suggests that the performance of our policy  gets closer to the one of affine policy as long as we add more budgets constraints. While affine policy performs better than our policy for budget of uncertainty sets, we would like to note that this is not necessarily true for any polyhedral uncertainty set. In particular, we also test our policy with the generalized budget \eqref{set:polytope} and observe that our policy is significantly better than affine even when the set is polyhedral.}

\red{
Table \ref{tab:simu:genpoly} presents the results for the generalized budget set \eqref{set:polytope}. We observe that our piecewise affine policy outperforms affine policy both in terms of objective value and computation time. The gap increases as $m$ increases which implies a significant improvement over affine policy for large values of $m$. Furthermore, unlike the piecewise affine policy, computing an affine solution becomes challenging for large values of $m$.}

\red{
For the intersection of budget of uncertainty sets \eqref{set:inters:budget} that are not permutation invariant, we compute the dominating set (in particular $\beta$ and $\mb v$) using Algorithm \ref{algo2}. We report the average running time, $T_{\sf Alg1}$ of Algorithm \ref{algo2} which solves a sequence of MIPs in Tables \ref{tab:simu:2bud} and \ref{tab:simu:5bud}. 
We note that there is no need to solve MIPs optimally in Algorithm \ref{algo2}; one can stop when a feasible solution with  an objective value greater than $t$ is found. We observe  that the running time of Algorithm \ref{algo2} is reasonable as compared to that of affine policy. For example, the average running time of  Algorithm \ref{algo2} for $m=100$ and $L=5$ is $7$ min whereas affine policy takes  $10$ min in average. For large values of $m$ and a large number of budget constraints, the running time of Algorithm 1 might increase significantly and  exceed the computation time of affine policy. However, we would like to emphasize that  $\beta $ and $\mb v$ given by Algorithm \ref{algo2}  do not depend on the  parameters $(\mb A, \mb B, \mb c, \mb d)$ and only depend on the uncertainty set. Therefore, they can be computed offline and can be used to solve many instances of the problem parameters for the same uncertainty set.}


\section{Conclusion}
This paper introduces a new framework for designing piecewise affine policies (PAP)  for two-stage adjustable robust optimization with right-hand side uncertainty. The framework is based on approximating the uncertainty set ${\cal U}$ by a dominating simplex and constructing a PAP using the map from ${\cal U}$ to the dominating simplex. For the class of conic uncertainty sets including ellipsoids and norm-balls, our PAP performs significantly better, \emph{theoretically and computationally} than affine policy.
For \emph{general} uncertainty sets (particularly a ``budgeted" ${\cal U}$ or intersection of a small number of ``budget of uncertainty sets"), our PAP does not necessarily outperform affine policies, but while the latter may fail for large dimensional ${\cal U}$, the PAP scales well given the dominating set. It is an interesting open question whether a PAP can be designed that significantly improves over affine policy for budgeted uncertainty sets.



 \begin{table}[]
 \red{
\centering
\begin{tabular}{|l|l|l|l|l|l|l|l|c|c|}
\hline
$m$ & ${\sf Avg}$ &  ${\sf Max}$& ${\sf Min}$ &   $5\%$ & $10\%$ &  $25\%$  & $50\%$  & $T_{\sf p-aff}(s)$ & $T_{\sf aff}(s)$\\ \hline
10  & 0.955 & 1.006 & 0.875 & 1.003 & 0.988 & 0.971 & 0.960 & 0.001 & 0.221   \\\hline
20  & 1.120 & 1.168 & 1.076 & 1.152 & 1.141 & 1.132 & 1.122 & 0.002 & 0.948   \\\hline
30  & 1.218 & 1.251 & 1.180 & 1.243 & 1.238 & 1.225 & 1.221 & 0.003 & 2.753   \\\hline
40  & 1.288 & 1.328 & 1.238 & 1.318 & 1.312 & 1.299 & 1.291 & 0.006 & 6.479   \\\hline
50  & 1.349 & 1.382 & 1.319 & 1.375 & 1.370 & 1.357 & 1.349 & 0.009 & 14.678  \\\hline
60  & 1.399 & 1.429 & 1.366 & 1.418 & 1.415 & 1.408 & 1.398 & 0.013 & 32.323  \\\hline
70  & 1.443 & 1.472 & 1.454 & 1.460 & 1.457 & 1.451 & 1.440 & 0.019 & 58.605  \\\hline
80  & 1.485 & 1.509 & 1.485 & 1.505 & 1.499 & 1.491 & 1.482 & 0.033 & 107.898 \\\hline
90  & 1.523 & 1.549 & 1.527 & 1.539 & 1.532 & 1.530 & 1.525 & 0.040 & 200.134 \\\hline
100 & 1.557 & 1.578 & 1.560 & 1.574 & 1.570 & 1.564 & 1.557 & 0.081 & 564.772\\\hline
\end{tabular}%
\caption{Comparison on the performance and computation time of affine policy and our piecewise affine policy for the {\bf hypersphere uncertainty set}. For 50 instances, we compute $ \frac{z_{\sf{Aff}} ({\cal U})    }{z_{\sf{p-aff}}  ({\cal U})         }$ and present the average, $\min$, $\max$ ratios and the percentiles $5\%, 10\%, 25\%, 50\%$. Here, $T_{\sf p-aff}(s)$ denotes the running time for our piecewise affine policy and $T_{\sf aff}(s)$ denotes the running time for affine policy in seconds.
}
\label{tab:simu:sphere}
\bigskip
\begin{tabular}{|l|l|l|l|l|l|l|l|c|c|}
\hline
$m$ & ${\sf Avg}$ &  ${\sf Max}$& ${\sf Min}$ &   $5\%$ & $10\%$ &  $25\%$  & $50\%$  & $T_{\sf p-aff}(s)$ & $T_{\sf aff}(s)$\\ \hline
10  & 0.975 & 1.049 & 0.907 & 1.023 & 1.017 & 0.991 & 0.971 & 0.001 & 0.743    \\\hline
20  & 1.082 & 1.141 & 1.042 & 1.128 & 1.119 & 1.097 & 1.080 & 0.002 & 3.714    \\\hline
30  & 1.157 & 1.195 & 1.094 & 1.190 & 1.177 & 1.167 & 1.158 & 0.003 & 12.386   \\\hline
40  & 1.218 & 1.247 & 1.184 & 1.236 & 1.233 & 1.226 & 1.219 & 0.006 & 31.687   \\\hline
50  & 1.270 & 1.294 & 1.245 & 1.293 & 1.284 & 1.275 & 1.271 & 0.009 & 69.302   \\\hline
60  & 1.312 & 1.346 & 1.274 & 1.335 & 1.325 & 1.319 & 1.312 & 0.013 & 117.949  \\\hline
70  & 1.345 & 1.363 & 1.323 & 1.361 & 1.358 & 1.351 & 1.347 & 0.020 & 258.862  \\\hline
80  & 1.378 & 1.402 & 1.356 & 1.396 & 1.393 & 1.384 & 1.378 & 0.031 & 435.629  \\\hline
90  & 1.408 & 1.429 & 1.389 & 1.421 & 1.418 & 1.413 & 1.409 & 0.043 & 728.436  \\\hline
100 & 1.434 & 1.457 & 1.419 & 1.447 & 1.443 & 1.438 & 1.433 & 0.050 & 1033.174 \\\hline
\end{tabular}%
\caption{Comparison on the performance and computation time of affine policy and our piecewise affine policy for the {\bf $3$-norm ball uncertainty set}.}
\label{tab:simu:3ball}
\bigskip
\begin{tabular}{|l|l|l|l|l|l|l|l|c|c|}
\hline
$m$ & ${\sf Avg}$ &  ${\sf Max}$& ${\sf Min}$ &   $5\%$ & $10\%$ &  $25\%$  & $50\%$  & $T_{\sf p-aff}(s)$ & $T_{\sf aff}(s)$\\ \hline
10  & 0.904 & 0.952 & 0.817 & 0.939 & 0.932 & 0.918 & 0.905 & 0.001 & 0.728   \\\hline
20  & 1.028 & 1.058 & 0.992 & 1.051 & 1.044 & 1.036 & 1.031 & 0.002 & 3.462   \\\hline
30  & 1.115 & 1.144 & 1.095 & 1.132 & 1.128 & 1.122 & 1.115 & 0.003 & 10.896  \\\hline
40  & 1.174 & 1.190 & 1.161 & 1.184 & 1.183 & 1.177 & 1.174 & 0.005 & 29.209  \\\hline
50  & 1.226 & 1.244 & 1.204 & 1.240 & 1.235 & 1.232 & 1.227 & 0.009 & 70.099  \\\hline
60  & 1.266 & 1.278 & 1.255 & 1.275 & 1.274 & 1.269 & 1.267 & 0.013 & 123.518 \\\hline
70  & 1.303 & 1.311 & 1.292 & 1.310 & 1.309 & 1.305 & 1.303 & 0.019 & 267.450 \\\hline
80  & 1.335 & 1.345 & 1.328 & 1.341 & 1.339 & 1.337 & 1.335 & 0.034 & 458.791 \\\hline
90  & 1.363 & 1.372 & 1.353 & 1.370 & 1.369 & 1.366 & 1.363 & 0.044 & 701.262 \\\hline
100 & 1.387 & 1.395 & 1.381 & 1.392 & 1.391 & 1.389 & 1.387 & 0.056 & 967.773 \\\hline
\end{tabular}%
\caption{Comparison on the performance and computation time of affine policy and our piecewise affine policy for the {\bf $3/2$-norm ball uncertainty set.}}
\label{tab:simu:small:ball}}
\end{table}


\begin{table}[]
\red{
\centering
\begin{tabular}{|l|l|l|l|l|l|l|l|c|c|}
\hline
$m$ & ${\sf Avg}$ &  ${\sf Max}$& ${\sf Min}$ &   $5\%$ & $10\%$ &  $25\%$  & $50\%$  & $T_{\sf p-aff}(s)$ & $T_{\sf aff}(s)$\\ \hline
10  & 0.906 & 0.989 & 0.766 & 0.986 & 0.974 & 0.957 & 0.915 & 0.001 & 0.014   \\\hline
20  & 0.897 & 0.963 & 0.780 & 0.957 & 0.951 & 0.939 & 0.916 & 0.002 & 0.207   \\\hline
30  & 0.891 & 0.961 & 0.765 & 0.957 & 0.945 & 0.923 & 0.906 & 0.004 & 0.803   \\\hline
40  & 0.882 & 0.954 & 0.753 & 0.950 & 0.946 & 0.928 & 0.900 & 0.006 & 2.997   \\\hline
50  & 0.899 & 0.954 & 0.763 & 0.950 & 0.947 & 0.937 & 0.914 & 0.011 & 11.687  \\\hline
60  & 0.879 & 0.956 & 0.772 & 0.953 & 0.948 & 0.932 & 0.896 & 0.015 & 26.760  \\\hline
70  & 0.887 & 0.958 & 0.911 & 0.951 & 0.950 & 0.936 & 0.909 & 0.020 & 71.167  \\\hline
80  & 0.882 & 0.954 & 0.768 & 0.951 & 0.946 & 0.937 & 0.902 & 0.047 & 147.376 \\\hline
90  & 0.890 & 0.953 & 0.765 & 0.950 & 0.949 & 0.936 & 0.917 & 0.039 & 220.809 \\\hline
100 & 0.886 & 0.955 & 0.750 & 0.946 & 0.943 & 0.931 & 0.900 & 0.066 & 397.981\\\hline
\end{tabular}%
\caption{Comparison on the performance and computation time of affine policy and our piecewise affine policy for the {\bf budget of uncertainty set }with a budget $k=c\sqrt{m}$ where for each instance we generate $c$ uniformly from $[1,2]$.}
\label{tab:simu:budget}
\bigskip
\begin{tabular}{|l|l|l|l|l|l|l|l|l|l|l|}
\hline
$m$ & ${\sf Avg}$ &  ${\sf Max}$& ${\sf Min}$ &   $5\%$ & $10\%$ &  $25\%$  & $50\%$  & $T_{\sf p-aff}(s)$ & $T_{\sf Alg1}(s)$ & $T_{\sf aff}(s)$\\ \hline
10  & 0.814 & 0.881 & 0.700 & 0.861 & 0.851 & 0.833 & 0.821 & 0.002 & 0.191 & 0.013   \\\hline
20  & 0.805 & 0.866 & 0.716 & 0.850 & 0.838 & 0.825 & 0.807 & 0.016 & 0.723 & 0.227   \\\hline
30  & 0.770 & 0.847 & 0.701 & 0.827 & 0.808 & 0.787 & 0.773 & 0.091 & 0.386 & 0.931   \\\hline
40  & 0.801 & 0.839 & 0.702 & 0.832 & 0.828 & 0.814 & 0.810 & 0.270 & 1.399 & 3.731   \\\hline
50  & 0.781 & 0.825 & 0.726 & 0.818 & 0.814 & 0.803 & 0.784 & 0.656 & 2.081 & 12.056  \\\hline
60  & 0.805 & 0.841 & 0.752 & 0.829 & 0.824 & 0.817 & 0.811 & 1.406 & 4.093 & 32.695  \\\hline
70  & 0.789 & 0.839 & 0.706 & 0.820 & 0.809 & 0.802 & 0.795 & 2.595 & 1.798 & 80.342  \\\hline
80  & 0.774 & 0.844 & 0.725 & 0.825 & 0.816 & 0.789 & 0.770 & 4.484 & 5.096 & 163.257 \\\hline
90  & 0.807 & 0.838 & 0.756 & 0.832 & 0.828 & 0.818 & 0.807 & 7.628 & 8.734 & 354.598 \\\hline
100 & 0.790 & 0.821 & 0.750 & 0.817 & 0.812 & 0.801 & 0.791 & 5.235 & 6.391 & 646.136\\\hline
\end{tabular}%
\caption{Comparison on the performance and computation time of affine policy and our piecewise affine policy for the {\bf intersection of 2 budget of uncertainty sets} \eqref{set:inters:budget}.}
\label{tab:simu:2bud}
\bigskip
\begin{tabular}{|l|l|l|l|l|l|l|l|l|l|l|}
\hline
$m$ & ${\sf Avg}$ &  ${\sf Max}$& ${\sf Min}$ &   $5\%$ & $10\%$ &  $25\%$  & $50\%$  & $T_{\sf p-aff}(s)$ & $T_{\sf Alg1}(s)$ & $T_{\sf aff}(s)$\\ \hline
10  & 0.869 & 0.932 & 0.824 & 0.920 & 0.910 & 0.884 & 0.871 & 0.002 & 0.043   & 0.015   \\\hline
20  & 0.852 & 0.924 & 0.795 & 0.909 & 0.893 & 0.870 & 0.852 & 0.021 & 0.058   & 0.309   \\\hline
30  & 0.864 & 0.898 & 0.820 & 0.888 & 0.880 & 0.872 & 0.865 & 0.100 & 0.343   & 1.024   \\\hline
40  & 0.856 & 0.896 & 0.802 & 0.883 & 0.882 & 0.874 & 0.861 & 0.290 & 0.464   & 4.010   \\\hline
50  & 0.857 & 0.891 & 0.794 & 0.891 & 0.886 & 0.876 & 0.861 & 0.706 & 3.546   & 12.535  \\\hline
60  & 0.880 & 0.900 & 0.860 & 0.894 & 0.892 & 0.885 & 0.881 & 1.471 & 18.474  & 33.693  \\\hline
70  & 0.873 & 0.896 & 0.809 & 0.894 & 0.890 & 0.882 & 0.878 & 2.800 & 13.125  & 82.961  \\\hline
80  & 0.858 & 0.889 & 0.825 & 0.886 & 0.881 & 0.872 & 0.858 & 4.809 & 21.780  & 167.753 \\\hline
90  & 0.859 & 0.890 & 0.818 & 0.885 & 0.881 & 0.877 & 0.866 & 8.004 & 144.808 & 344.924 \\\hline
100 & 0.885 & 0.902 & 0.865 & 0.900 & 0.896 & 0.893 & 0.888 & 5.821 & 459.436 & 632.483 \\\hline
\end{tabular}%
\caption{Comparison on the performance and computation time of affine policy and our piecewise affine policy for the {\bf intersection of 5 budget of uncertainty sets}  \eqref{set:inters:budget}.}
\label{tab:simu:5bud}}
\end{table}`

\begin{table}[]
\red{
\centering
\begin{tabular}{|l|l|l|l|l|l|l|l|c|c|}
\hline
$m$ & ${\sf Avg}$ &  ${\sf Max}$& ${\sf Min}$ &   $5\%$ & $10\%$ &  $25\%$  & $50\%$  & $T_{\sf p-aff}(s)$ & $T_{\sf aff}(s)$\\ \hline
10  & 1.015 & 1.067 & 0.983 & 1.053 & 1.045 & 1.025 & 1.006 & 0.001 & 0.046    \\\hline
20  & 1.107 & 1.159 & 1.100 & 1.147 & 1.142 & 1.127 & 1.106 & 0.003 & 0.840    \\\hline
30  & 1.148 & 1.214 & 1.092 & 1.189 & 1.179 & 1.163 & 1.155 & 0.004 & 3.933    \\\hline
40  & 1.173 & 1.220 & 1.105 & 1.206 & 1.198 & 1.188 & 1.175 & 0.009 & 18.097   \\\hline
50  & 1.191 & 1.227 & 1.154 & 1.216 & 1.213 & 1.201 & 1.189 & 0.016 & 62.668   \\\hline
60  & 1.209 & 1.259 & 1.193 & 1.238 & 1.225 & 1.215 & 1.210 & 0.021 & 145.552  \\\hline
70  & 1.225 & 1.254 & 1.190 & 1.247 & 1.239 & 1.228 & 1.224 & 0.019 & 237.448  \\\hline
80  & 1.237 & 1.275 & 1.213 & 1.264 & 1.260 & 1.245 & 1.235 & 0.044 & 573.342  \\\hline
90  & 1.248 & 1.284 & 1.223 & 1.268 & 1.260 & 1.254 & 1.249 & 0.050 & 1168.928 \\\hline
100 & 1.257 & 1.274 & 1.240 & 1.271 & 1.268 & 1.261 & 1.257 & 0.053 & 1817.940\\\hline
\end{tabular}%
\caption{Comparison on the performance and computation time of affine policy and our piecewise affine policy for the {\bf generalized budget of uncertainty set \eqref{set:polytope}.}}
\label{tab:simu:genpoly}}
\end{table}

\newpage
\bibliographystyle{abbrv}
\bibliography{robust.bib}

\begin{thebibliography}{10}

\bibitem{ayoub2016decomposition}
J.~Ayoub and M.~Poss.
\newblock Decomposition for adjustable robust linear optimization subject to
  uncertainty polytope.
\newblock {\em Computational Management Science}, 13(2):219--239, 2016.

\bibitem{BNE10}
A.~Ben-Tal, L.~El~Ghaoui, and A.~Nemirovski.
\newblock {\em {Robust optimization}}.
\newblock Princeton University press, 2009.

\bibitem{Ben-Tal04}
A.~Ben-Tal, A.~Goryashko, E.~Guslitzer, and A.~Nemirovski.
\newblock {Adjustable robust solutions of uncertain linear programs}.
\newblock {\em Mathematical Programming}, 99(2):351--376, 2004.

\bibitem{BN98}
A.~Ben-Tal and A.~Nemirovski.
\newblock {Robust convex optimization}.
\newblock {\em Mathematics of Operations Research}, 23(4):769--805, 1998.

\bibitem{BN99}
A.~Ben-Tal and A.~Nemirovski.
\newblock {Robust solutions of uncertain linear programs}.
\newblock {\em Operations Research Letters}, 25(1):1--14, 1999.

\bibitem{Ben-Tal02}
A.~Ben-Tal and A.~Nemirovski.
\newblock {Robust optimization--methodology and applications}.
\newblock {\em Mathematical Programming}, 92(3):453--480, 2002.

\bibitem{BB15}
D.~Bertsimas and H.~Bidkhori.
\newblock On the performance of affine policies for two-stage adaptive
  optimization: a geometric perspective.
\newblock {\em Mathematical Programming}, 153(2):577--594, 2015.

\bibitem{BBC08}
D.~Bertsimas, D.~Brown, and C.~Caramanis.
\newblock Theory and applications of robust optimization.
\newblock {\em SIAM review}, 53(3):464--501, 2011.

\bibitem{bertsimas2010finite}
D.~Bertsimas and C.~Caramanis.
\newblock Finite adaptability in multistage linear optimization.
\newblock {\em Automatic Control, IEEE Transactions on}, 55(12):2751--2766,
  2010.

\bibitem{BD15}
D.~Bertsimas and I.~Dunning.
\newblock Multistage robust mixed integer optimization with adaptive
  partitions.
\newblock {\em Under Review}, 2015.

\bibitem{bertsimas2015design}
D.~Bertsimas and A.~Georghiou.
\newblock Design of near optimal decision rules in multistage adaptive
  mixed-integer optimization.
\newblock {\em Operations Research}, 63(3):610--627, 2015.

\bibitem{BG10}
D.~Bertsimas and V.~Goyal.
\newblock {On the Power and Limitations of Affine Policies in Two-Stage
  Adaptive Optimization}.
\newblock {\em Mathematical Programming}, 134(2):491--531, 2012.

\bibitem{BGS10}
D.~Bertsimas, V.~Goyal, and X.~Sun.
\newblock A geometric characterization of the power of finite adaptability in
  multistage stochastic and adaptive optimization.
\newblock {\em Mathematics of Operations Research}, 36(1):24--54, 2011.

\bibitem{BIP09}
D.~Bertsimas, D.~Iancu, and P.~Parrilo.
\newblock {Optimality of Affine Policies in Multi-stage Robust Optimization}.
\newblock {\em Mathematics of Operations Research}, 35:363--394, 2010.

\bibitem{BS03}
D.~Bertsimas and M.~Sim.
\newblock {Robust Discrete Optimization and Network Flows}.
\newblock {\em Mathematical Programming Series B}, 98:49--71, 2003.

\bibitem{BS04}
D.~Bertsimas and M.~Sim.
\newblock {The Price of Robustness}.
\newblock {\em Operations Research}, 52(2):35--53, 2004.

\bibitem{chen2008linear}
X.~Chen, M.~Sim, P.~Sun, and J.~Zhang.
\newblock A linear decision-based approximation approach to stochastic
  programming.
\newblock {\em Operations Research}, 56(2):344--357, 2008.

\bibitem{Dantzig55}
G.~Dantzig.
\newblock Linear programming under uncertainty.
\newblock {\em Management Science}, 1:197--206, 1955.

\bibitem{EL97}
L.~El~Ghaoui and H.~Lebret.
\newblock {Robust solutions to least-squares problems with uncertain data}.
\newblock {\em SIAM Journal on Matrix Analysis and Applications},
  18:1035--1064, 1997.

\bibitem{housni2017beyond}
O.~El~Housni and V.~Goyal.
\newblock Beyond worst-case: A probabilistic analysis of affine policies in
  dynamic optimization.
\newblock In I.~Guyon, U.~V. Luxburg, S.~Bengio, H.~Wallach, R.~Fergus,
  S.~Vishwanathan, and R.~Garnett, editors, {\em Advances in Neural Information
  Processing Systems 30}, pages 4759--4767. Curran Associates, Inc., 2017.

\bibitem{elhousni2015piecewise}
O.~El~Housni and V.~Goyal.
\newblock Piecewise static policies for two-stage adjustable robust linear
  optimization.
\newblock {\em Mathematical Programming}, pages 1--17, 2017.

\bibitem{FJMM07}
U.~Feige, K.~Jain, M.~Mahdian, and V.~Mirrokni.
\newblock {Robust combinatorial optimization with exponential scenarios}.
\newblock {\em Lecture Notes in Computer Science}, 4513:439--453, 2007.

\bibitem{GI03}
D.~Goldfarb and G.~Iyengar.
\newblock {Robust portfolio selection problems}.
\newblock {\em Mathematics of Operations Research}, 28(1):1--38, 2003.

\bibitem{ISS13}
D.~Iancu, M.~Sharma, and M.~Sviridenko.
\newblock Supermodularity and affine policies in dynamic robust optimization.
\newblock {\em Operations Research}, 61(4):941--956, 2013.

\bibitem{KW94}
P.~Kall and S.~Wallace.
\newblock {\em {Stochastic programming}}.
\newblock Wiley New York, 1994.

\bibitem{postek2016multistage}
K.~Postek and D.~d. Hertog.
\newblock Multistage adjustable robust mixed-integer optimization via iterative
  splitting of the uncertainty set.
\newblock {\em INFORMS Journal on Computing}, 28(3):553--574, 2016.

\bibitem{Prekopa95}
A.~Pr{\'e}kopa.
\newblock {\em {Stochastic programming}}.
\newblock Kluwer Academic Publishers, Dordrecht, Boston, 1995.

\bibitem{Shapiro08}
A.~Shapiro.
\newblock {Stochastic programming approach to optimization under uncertainty}.
\newblock {\em Mathematical Programming, Series B}, 112(1):183--220, 2008.

\bibitem{SDR09}
A.~Shapiro, D.~Dentcheva, and A.~Ruszczy{\'n}ski.
\newblock {\em {Lectures on stochastic programming: modeling and theory}}.
\newblock Society for Industrial and Applied Mathematics, 2009.

\bibitem{SA73}
A.~Soyster.
\newblock Convex programming with set-inclusive constraints and applications to
  inexact linear programming.
\newblock {\em Operations research}, 21(5):1154--1157, 1973.

\bibitem{zeng2011solving}
B.~Zeng.
\newblock Solving two-stage robust optimization problems by a
  constraint-and-column generation method.
\newblock {\em University of South Florida, FL, Tech. Rep}, 2011.

\end{thebibliography}

%
%

\appendix
\begin{appendix}
\section{Proof of Theorem \ref{thm:ineq}}\label{apx-proofs:thm:ineq}

\begin{proof}
Let $(\hat{\mb{x}},\mb{\hat{y}}(\hat{\mb{h}}), \hat{\mb{h}}\in\hat{\cal U})$ be an optimal solution for $z_{\sf AR}(\hat{\cal U})$. For each $\mb{h}\in{\cal U}$, let $\tilde{\mb{y}}(\mb{h})= \hat{\mb{y}}(\hat{\mb{h}})$ where $\hat{\mb{h}}\in\hat{\cal U}$ dominates $\mb{h}$. Therefore, for any $\mb{h}\in{\cal U}$,
\[\mb{A}\hat{\mb{x}}+\mb{B}\tilde{\mb{y}}(\mb{h}) = \mb{A}\hat{\mb{x}}+\mb{B}\hat{\mb{y}}(\hat{\mb{h}})\geq\hat{\mb{h}}\geq\mb{h},\]
i.e.,  $(\hat{\mb{x}},\tilde{\mb{y}}(\mb{h}),\mb{h}\in{\cal U})$ is a feasible solution for $z_{\sf{AR}}({\cal U})$. Therefore,
\[z_{\sf AR}({\cal U})\leq\mb{c}^T\hat{\mb{x}}+\max_{\mb{h}\in{\cal U}}\mb{d}^T\tilde{\mb{y}}(\mb{h})\leq\mb{c}^T\hat{\mb{x}}+\max_{\hat{\mb{h}}\in\hat{\cal U}}\mb{d}^T\hat{\mb{y}}(\hat{\mb{h}})=z_{\sf{AR}}(\hat{\cal U}).\]
Conversely, let $(\mb{x}^*, \mb{y}^*(\mb{h}), \mb{h}\in{\cal U})$ be an optimal solution of $z_{\sf AR}({\cal U})$. Then, for any $\hat{\mb{h}}\in\hat{\cal U}$, since $ \frac{\hat{\mb{h}}}{\beta} \in{\cal U}$, we have,
\[\mb{A}\mb{x}^*+\mb{B}\mb{y}^*\left(\frac{\hat{\mb{h}}}{\beta}\right)\geq\frac{\hat{\mb{h}}}{\beta},\]
Therefore, $(\beta\mb{x}^*, \beta\mb{y}^*\left(\frac{\hat{\mb{h}}}{\beta}\right), \hat{\mb{h}} \in {\cal U})$ is feasible for $\Pi_{\sf AR}(\hat{\cal U})$. Therefore,
\[z_{\sf AR}(\hat{\cal U})\leq  \mb c^T  \beta  \mb x^* + \underset{ \mb{\hat{h}} \in \hat{\cal U}} \max \; \mb d^T   \beta  \mb{y^*}\left(\frac{\mb{\hat{h}}}{\beta}\right) \leq \beta \cdot \left( \mb c^T   \mb x^* + \underset{ \mb{h} \in {\cal U}} \max \; \mb d^T    \mb{y^*(h)} \right) = \beta \cdot z_{\sf{AR}}(\cal U).\]
 \end{proof}

\section{Proof of Lemma \ref{lem:ineq}}\label{apx-proofs:lem:ineq}
 \begin{proof}
a)
Suppose there exists $ \beta $ and $ \mb v \in {\cal U}$ such that
$ \;  \hat{\cal U} = \beta \cdot {\sf conv } \left( \mb e_1,\ldots,\mb e_m, \mb v \right)$ dominates $\cal U $.
Consider $\mb h \in {\cal U }$. Since  $\hat{\cal U}$ dominates  ${\cal U}$, there exists $ \alpha_1, \alpha_2,\ldots,\alpha_{m+1} \geq 0$ with $\alpha_1 + \ldots + \alpha_{m+1} =1$ such that \begin{equation} \label{ax1}
 h_i \leq \beta \left( \alpha_i + \alpha_{m+1}  v_i  \right), \; \forall i =1,\ldots,m.
\end{equation}
Let
$$ I (\mb h )= \left\{ i \in [m] \; \bigg\vert \;  h_i - \beta v_i \geq 0 \right\}.$$ Then,
 \begin{align*}
\sum_{i=1}^m \left( h_i -\beta v_i \right)^{+} &= \underset {i \in I (\mb h ) } \sum  h_i - \beta \underset {i \in I (\mb h ) } \sum v_i \\
& \leq  \sum_{i\in I (\mb h )}  \beta \left( \alpha_i + \alpha_{m+1}  v_i  \right)  - \beta \underset {i \in I (\mb h )} \sum v_i\\
& = \beta \sum_{i\in I (\mb h )}   \alpha_i +  \left( \alpha_{m+1} -1\right) \beta \underset {i \in I (\mb h ) } \sum  v_i\\
& \leq \beta, \\
\end{align*}
where the first inequality follows from \eqref{ax1} and the last inequality holds because $ \alpha_{m+1} -1 \leq 0$, $ v_i \geq 0$ , $\beta \geq 0$ and $ \sum_{i\in I (\mb h )}   \alpha_i \leq 1$.
We conclude that
  $$  \frac{1}{\beta} \sum_{i=1}^m \left( h_i -\beta v_i \right)^{+} \leq 1.$$

\red{b) Now, suppose there exists $ \beta $ and $ \mb v \in {\cal U}$ such that
$ \;  \hat{\cal U} = \beta \cdot {\sf conv } \left( \mb e_1,\ldots,\mb e_m, \mb v \right)$ dominates $\cal U $.
For any $\mb h \in {\cal U}$, let
\[ \mb{\hat h} =  \sum_{i=1}^m \left( h_i -\beta v_i  \right)^{+} \mb e_i + \beta \mb v.\]
Then for all $i=1,\ldots,m$,
\begin{align*}
\hat{h}_i  &=  \left( h_i -\beta v_i  \right)^{+} + \beta  v_i \\
&\geq \left( h_i -\beta v_i  \right) + \beta  v_i \geq h_i .
\end{align*}
Therefore, $\mb{\hat h}$ dominates $\mb h$. Moreover,
\[ \mb{\hat h} = 2\beta \left( \sum_{i=1}^m \frac{\left( h_i -\beta v_i  \right)^{+}}{2 \beta} \mb e_i + \frac{1}{2}\mb v \right) \in 2\beta \cdot {\sf conv } \left( \mb 0, \mb e_1,\ldots,\mb e_m, \mb v \right), \]
because
$$  \frac{1}{\beta} \sum_{i=1}^m \left( h_i - \beta v_i  \right)^{+}  \leq 1. $$
Therefore, $2\beta \cdot {\sf conv } \left(\mb 0,  \mb e_1,\ldots,\mb e_m, \mb v \right)  $ dominates ${\cal U} $ and consequently \\$2\beta \cdot {\sf conv } \left( \mb e_1,\ldots,\mb e_m, \mb v \right)  $ dominates ${\cal U} $ as well.}
     \end{proof}

\section{Proof of Lemma \ref{lem:gamma(k)}}\label{apx-proofs:em:gamma(k)}

\begin{proof}
Suppose $ k \in [m]$. Let us consider
$$ \mb h  \in   \underset{\mb h \in {\cal U}}{  \sf argmax} \sum_{i=1}^k h_i .$$
Without loss of generality, we can suppose that $h_i =0$ for $ i =k+1,\ldots, m$. Denote, ${\cal S}_k$ the set of permutations of $\{ 1,2,\ldots,k\}$. We define $\mb h ^{\sigma} \in \mathbb{R}_+^m$ such that $h^{\sigma}_i = h_{\sigma(i) }$ for $i=1, \ldots ,k$ and $h^{\sigma}_i =0$ otherwise.
Since ${\cal U}$ is a permutation invariant set, we have $\mb h ^{\sigma} \in  {\cal U}$ for any $\sigma \in {\cal S}_k$. The convexity of ${\cal U}$ implies that
$$ \frac{1}{k!} \sum_{ \sigma \in {\cal S}_k} \mb h^{\sigma} \in {\cal U}.$$
We have,
$$  \sum_{ \sigma \in {\cal S}_k}  h^{\sigma}_i = \left\{
    \begin{array}{ll}
        (k-1)! \cdot \sum_{j=1}^k h_j & \mbox{if } i=1,\ldots,k \\
        0 & \mbox{otherwise,}
    \end{array}
\right.$$
and $\sum_{j=1}^k h_j = k \cdot \gamma(k)$ by definition. Therefore, $$ \frac{1}{k!} \sum_{ \sigma \in {\cal S}_k} \mb h^{\sigma} =
 \gamma (k) \cdot \sum_{i=1}^k \mb e_i \in {\cal U}.$$
 \end{proof}

\section{Proof of Lemma \ref{lem:struct:sol}}\label{apx-proofs:lem:struct:sol}

\begin{proof}
Consider, $ \tilde{\mb h} \in {\cal U} $ an optimal solution  for the maximization problem in \eqref{eq:ineq:perm} for fixed $\beta$. We will construct $ \mb h^* \in {\cal U}$ another optimal solution of \eqref{eq:ineq:perm} that verifies the properties in the lemma. First, denote
$ I =   \{ i \; \vert \;  \tilde{h}_i > \beta \gamma\}   $ and  $\vert I\vert =k$. Since, ${\cal U}$ is permutation invariant, we can suppose without loss of generality that $I =\{1,2,\ldots,k \}$. We define,
$$  h^*_i = \left\{
    \begin{array}{ll}
       \gamma(k) & \mbox{if } i=1,\ldots,k \\
        0 & \mbox{otherwise.}
    \end{array}
\right.$$
From Lemma \ref{lem:gamma(k)}, we have $\mb h^* \in {\cal U}$. Moreover,
\begin{align*}
\sum_{i=1}^m ( \tilde{h}_i - \beta \gamma)^+  = \sum_{i=1}^k \tilde{h}_i  - \beta \gamma k  & \leq k  \cdot\gamma(k) - \beta \gamma k \\
& =  \sum_{i=1}^k ( \gamma(k) - \beta \gamma )  =  \sum_{i=1}^k ( h^*_i- \beta \gamma ) \\
& \leq  \sum_{i=1}^k ( h^*_i- \beta \gamma )^+ =  \sum_{i=1}^m ( h^*_i- \beta \gamma )^+\\
\end{align*}
where the first inequality follows from the definition of the coefficients $\gamma (.)$.  Therefore, $\mb{h}^*$ and $ \tilde{\mb h}$ have the same objective value in \eqref{eq:ineq:perm} and consequently $\mb{h}^*$ is also optimal for the maximization problem \eqref{eq:ineq:perm}. Moreover, from the first inequality, we have $ \gamma(k) - \beta \gamma  >0 $, i.e.,  $\big\vert  \{ i \; \vert \;  h_i^* > \beta \gamma\}    \big\vert  =k.$ Therefore, $\mb{h}^*$ verifies the properties of the lemma.
 \end{proof}

\section{Proof of Proposition \ref{prop:ellipsoid}}\label{apx-proofs:prop:ellipsoid}

\begin{proof}
To prove that $\hat{\cal U}$ dominates ${\cal U}$, it is sufficient to take $\mb{h}$ in the boundaries of $ {\cal U}$, i.e.,
\begin{equation} \label{apax}
a \sum_{i=1}^m h_i \sum_{j=1}^m h_j + (1-a) \sum_{i=1}^m h_i^2 =1 ,
\end{equation}
and find $ \alpha_1, \alpha_2,\ldots,\alpha_{m+1}$ nonnegative reals with $\sum_{i=1}^{m+1} \alpha_i =1$ such that for all $ i \in [m],$ $$ \; h_i \leq \beta \left( \alpha_i + \gamma \alpha_{m+1}\right).$$
By taking all $h_i$ equal in \eqref{apax}, we get
$$ \gamma = \frac{1}{ \sqrt{\left( am^2+(1-a)m \right)}  } .   $$
We choose  for $ i \in [m]$, $$ \alpha_i = \frac{1}{2} \left( (1-a)h_i^2+ a h_i \sum_{j=1}^m h_j   \right)     $$ and $\alpha_{m+1}= \frac{1}{2}.$ First, we have $\sum_{i=1}^{m+1} \alpha_i =1$ and for all $ i \in [m]$,
\begin{align*}
 \beta \left( \alpha_i + \gamma \alpha_{m+1}\right) &=  \frac{\beta}{2}  \left( (1-a)h_i^2+ a h_i \sum_{j=1}^m h_j  + \frac{1}{ \sqrt{am^2+(1-a)m}  } \right) \\
 & \geq \frac{\beta}{2}  \left( (1-a)h_i^2   + \frac{1}{ \sqrt{am^2+(1-a)m}  } + a h_i \right) \\
  & \geq \frac{\beta}{2}  \left( 2 \left(  \frac{(1-a)}{\sqrt{am^2+(1-a)m}}      \right)^{\frac{1}{2}}h_i+ a h_i \right) =h_i \\
\end{align*}
where the first inequality holds because $\sum_{j=1}^m h_j \geq 1$ which is a direct consequence of $ \mb h^T \Sigma \mb h =1$ and $ a \leq 1$. The second one follows from \red{the inequality of arithmetic and geometric means (AM-GM inequality).}
Finally, we can verify by case analysis on the values of $a$ that $$  \left(  \frac{a}{2}+ \frac{(1-a)^{\frac{1}{2}}}{\left( am^2+(1-a)m\right)^{\frac{1}{4}}}\right)^{-1}= O \left(  m^{\frac{2}{5}} \right) .$$
In fact, denote $H(m)= \left(  \frac{a}{2}+ \frac{(1-a)^{\frac{1}{2}}}{\left( am^2+(1-a)m\right)^{\frac{1}{4}}}\right)^{-1} = O \left(  a+ \frac{1}{\left( am^2+m\right)^{\frac{1}{4}}}\right)^{-1}$

{\bf Case1: $a=O(\frac{1}{m})$.} We have $\left( am^2+m \right)^{\frac{1}{4}} = O(m^{\frac{1}{4}})$. Then $H(m)=O(m^{\frac{1}{4}})=O(m^{\frac{2}{5}})$.

{\bf Case2: $a=\Omega(m^{\frac{-2}{5}})$.} We have  $H(m)=O(a^{-1})=O(m^{\frac{2}{5}})$.

{\bf Case3: $a=O(m^{\frac{-2}{5}} )$ and $a=\Omega(\frac{1}{m})$.} We have $\left( am^2+m \right)^{\frac{1}{4}} = O(m^{\frac{2}{5}})$. Then, $$ a+ \frac{1}{\left( am^2+m\right)^{\frac{1}{4}}}= \Omega(\frac{1}{m})+\Omega(m^{\frac{-2}{5}})=\Omega(m^{\frac{-2}{5}}).$$ Therefore, $H(m)=O(m^{\frac{2}{5}})$.
 \end{proof}

\section{Proof of Proposition \ref{prop:Budget of uncertainty set}}\label{apx-proofs:prop:Budget of uncertainty set}

\begin{proof}
To prove that $\hat{ \cal U}$ dominates ${ \cal U}$, it is sufficient to take $\mb h$ in the boundaries of $ {\cal U}$, i.e.,  $\sum_{i=1}^m h_i =k$ and find $ \alpha_1, \alpha_2,\ldots,\alpha_{m+1}$ non-negative reals with $\sum_{i=1}^{m+1} \alpha_i =1$ such that for all $ i \in [m],$ $$ \; h_i \leq \beta \left( \alpha_i + \frac{k}{m}\alpha_{m+1}\right).$$
{\em{First case:}} If $\beta=k$, we choose  $\alpha_i = \frac{h_i}{k}$ for $ i \in [m]$ and $\alpha_{m+1}= 0.$ We have $\sum_{i=1}^{m+1} \alpha_i =1$ and for all $ i \in [m]$,
$$ \beta \left( \alpha_i + \frac{k}{m}\alpha_{m+1}\right) =  k \frac{h_i}{k} \geq h_i .$$
{\em{Second case:}} If $\beta=\frac{m}{k}$, we choose  $\alpha_i = 0 $ for $ i \in [m]$ and $\alpha_{m+1}= 1.$ We have $\sum_{i=1}^{m+1} \alpha_i =1$ and for all $ i \in [m]$,
$$ \beta \left( \alpha_i + \frac{k}{m}\alpha_{m+1}\right) =  1 \geq h_i .$$
 \end{proof}
\red{
\section{Proof of Lemma \ref{lem:tight-comp-paff}}\label{apx-proofs:lem:tight-comp-paff}
\begin{proof}
Consider the following simplex
$$ \hat{\cal U} = {\sf{conv}} \left(\mb e_1 , \ldots ,\mb e_m, \frac{1}{\sqrt{m}} \mb e        \right)$$
It is clear that $ \hat{\cal U}$ dominates  ${\cal U}$ since $\frac{1}{\sqrt{m}} \mb e$ dominates all the extreme points $ \mb \nu_j$ for $j \in [N]$.
Moreover, by the convexity of $ \cal U$,  we have $ \frac{1}{N} \sum_{j=1}^N \mb {\nu}_j = \frac{\binom {m-1} {r-1}}{\sqrt{m}\binom m r}   \mb e = \frac{r}{m\sqrt{m}} \mb e \in {\cal U}$. Denote $\beta = \frac{m}{r } $. Hence, for all $i \in [m]$
$$ \mb e_i = \beta  \underbrace{\left( \frac{1}{\beta} \cdot \mb e_i + (1- \frac{1}{\beta}) \cdot \mb 0\right) }_{\in {\cal U}} \qquad \text{and} \qquad \frac{1}{\sqrt{m}} \mb e = \beta \cdot  \underbrace{\frac{r}{m\sqrt{m}} \mb e}_{\in \cal U}. $$
Therefore, $ \hat{\cal U} \subseteq \beta \cdot {\cal U}$ and from Theorem \ref{thm:ineq}, we conclude that our policy gives a $\beta$-approximation to the adjustable problem ~\eqref{eqn:ar} where $\beta = \frac{m}{\lceil m- \sqrt{m \rceil}  }=O (1+ \frac{1}{\sqrt{m}})$.
 \end{proof}
\section{Proof of Lemma \ref{lem:tight-comp-aff}}\label{apx-proofs:lem:tight-comp-aff}
\begin{proof}
First, let us prove that $z_{\sf AR}({\cal U}) \leq 1$. It is sufficient to define an adjustable solution only for the extreme points of ${\cal U}$ because the constraints are linear. We define the following solution for all $i=1,\ldots,m$ and for all $j=1,\ldots,N$
$$  \mb x = \mb 0 , \qquad \mb y ( \mb 0) = \mb 0, \qquad  \mb y ( \mb e_i) = \mb e_i, \qquad  \mb y ( \mb \nu_j) = \frac{1}{m} \mb e.$$
We have $\mb B \mb y ( \mb 0) = \mb 0$. For $i \in [m]$
$$\mb B \mb y ( \mb e_i) = \mb e_i + \frac{1}{\sqrt{m}} ( \mb e - \mb e_i) \geq  \mb e_i$$
and for $j \in [N]$
$$\mb B \mb y (\mb \nu_j) = \frac{1}{m} \mb B \mb e = \left( \frac{1}{m}+ \frac{m-1}{m \sqrt{m}} \right) \mb e \geq \frac{1}{\sqrt{m}} \mb e    \geq \mb \nu_j.$$
Therefore, the solution defined above is feasible. Moreover, the cost of our feasible solution is $1$ because for all $i \in [m]$ and $j \in [N]$,   we have
$$ \mb d^T \mb y ( \mb e_i)= \mb d^T \mb y ( \mb \nu_j)= 1.$$
Hence, $z_{\sf AR}({\cal U}) \leq 1.$
Now, it is sufficient to prove that  $z_{\sf Aff}({\cal U})= \Omega ( \sqrt{m})$. First, $\tilde{\mb x}= \frac{1}{\sqrt{m}} \mb e $ and $\mb y( \mb h)= \mb 0$ for any $\mb h \in {\cal U}$ is a feasible static solution (which is a special case of an affine solution).  In fact, $$\mb A \tilde{\mb x}= \frac{1}{\sqrt{m}} \mb A \mb e = \left( \frac{1}{\sqrt{m}}+ \frac{m-1}{m} \right) \mb e \geq \mb e \geq \mb h      \qquad \forall \mb h \in {\cal U}$$
where the last inequality holds because $ {\cal U} \subseteq [0,1]^m$. Moreover, the cost of this static solution is
$$ \mb c^T \tilde{\mb x} = \frac{\sqrt{m}}{15}.$$
Hence,
\begin{equation} \label{contradic}
z_{\sf Aff}({\cal U}) \leq  \frac{\sqrt{m}}{15}.
\end{equation}
Our instance is "a permuted instance", i.e. ${\cal U}$ is  permutation invariant, $\mb A$ and $\mb B$ are symmetric and $\mb c$ and $\mb d$ are proportional to $\mb e$. Hence,
from Lemma 8 and Lemma 7 in Bertsimas and Goyal \cite{BG10},  for any optimal solution $\mb x^*_{\sf Aff}, \mb y^*_{\sf Aff}( \mb h)$ of the affine problem, we can construct another optimal affine solution that is "symmetric" and have the same  stage cost. In particular, there exists an optimal solution for the affine problem  of the following form $ \mb x= \alpha \mb e$, $ \mb y( \mb h) = \mb P \mb h + \mb q$ for $\mb h \in {\cal U}$ where
\begin{equation} \label{matrix:P}
 \mb P= \left(
\begin{matrix}
\theta &  \mu & \ldots & \mu \\
\mu &  \theta & \ldots & \mu \\
\vdots & \vdots & \ddots & \vdots\\
\mu  &   \mu    &\ldots & \theta
\end{matrix}
\right)
\end{equation}
 $ \mb q = \lambda \mb e$, $\mb c^T \mb x= \mb c^T \mb x^*_{\sf Aff}$  and  $\max_{ \mb h \in {\cal U}} \mb d^T \mb y(\mb h)= \max_{ \mb h \in {\cal U}} \mb d^T \mb y^*_{\sf Aff}(\mb h)$.  We have $ \mb x \geq \mb 0$ and $ \mb y(\mb 0) = \lambda \mb e \geq \mb 0$ hence
\begin{equation}\label{eq:lambda}
\lambda \geq 0 \qquad \text{and}   \qquad \alpha \geq 0.
\end{equation}
\noindent
\underline{\bf Claim:} $\alpha \geq \frac{1}{24 \sqrt{m}}$
For a sake of contradiction, suppose that $ \alpha > \frac{1}{24 \sqrt{m}}$. We know that
\begin{equation}\label{eq:lem:1}
z_{\sf Aff}({\cal U}) \geq \mb c^T \mb x + \mb d^T \mb y( \mb 0) =  \frac{\alpha}{15} m + \lambda m.
\end{equation}
{\bf Case 1:} If $\lambda \geq \frac{1}{12 \sqrt{m}}$, then from \eqref{eq:lem:1} and $\alpha \geq 0$, we have $z_{\sf Aff}({\cal U}) \geq \frac{\sqrt{m}}{12}$. Contradiction with \eqref{contradic}.\\
{\bf Case 2:} If $ \lambda \leq \frac{1}{12 \sqrt{m}}$. We have
$$ \mb y( \mb e_1) = ( \theta+ \lambda ) \mb e_1 + ( \mu+ \lambda) ( \mb e -  \mb e_1).$$
By feasibility of the solution, we have $\mb A \mb x+ \mb B \mb y ( \mb e_1) \geq \mb e_1$, hence
$$ \theta+ \lambda +  \alpha \left( \frac{m-1}{\sqrt{m}} +1 \right)+\frac{1}{\sqrt{m}} (m-1)(\mu+ \lambda ) \geq 1$$
Therefore $\theta+ \lambda +  \alpha \left( \frac{m-1}{\sqrt{m}} +1 \right) \geq \frac{1}{2}$ or $\frac{1}{\sqrt{m}} (m-1)(\mu + \lambda) \geq \frac{1}{2}$.\\
{\bf Case 2.1:} Suppose $\frac{1}{\sqrt{m}} (m-1)(\mu + \lambda ) \geq \frac{1}{2}$.
Therefore,
$$ z_{\sf Aff}({\cal U}) \geq \mb d^T \mb y (\mb e_1) = \theta+ \lambda + (m-1)(\mu+ \lambda) \geq \frac{\sqrt{m}}{2}. \; \; (\text{Contradiction with \eqref{contradic}})$$
where the last inequality holds because $\theta+ \lambda \geq 0 $ as $\mb y( \mb e_1) \geq \mb 0$.\\
{\bf Case 2.2:} Now suppose we have the other inequality i.e. $\theta+ \lambda +  \alpha \left( \frac{m-1}{\sqrt{m}} +1 \right) \geq \frac{1}{2}$. Recall that we have $ \lambda \leq \frac{1}{12\sqrt{m}}$ and we know that $ \alpha < \frac{1}{24 \sqrt{m}}$. Therefore,
$$ \theta \geq \frac{1}{2}- \frac{1}{12\sqrt{m}} -  \frac{1}{24\sqrt{m}}\left( \frac{m-1}{\sqrt{m}} +1 \right)   = \frac{11}{24} - \frac{3}{24 \sqrt{m}} + \frac{1}{24m}
 \geq \frac{11}{24} - \frac{3}{24 }= \frac{1}{3}.$$
We have,
$$ \mb y ( \mb \nu_1 )  =  \frac{1}{\sqrt{m}}     \left( ( \theta + (r-1) \mu ) ( \mb e_1+ \ldots \mb e_r)   + r \mu ( \mb e -( \mb e_1+ \ldots \mb e_r)) \right) + \lambda \mb e. $$
In particular we have ,
\begin{align} \label{eq:up} \nonumber
 z_{\sf Aff}({\cal U}) \geq \mb d^T \mb y (\mb \nu_1) &= \frac{r}{\sqrt{m}} (  \theta +(m-1) \mu) + \lambda m \\
 & \geq    \frac{r}{\sqrt{m}} \left(  \frac{1}{3} + (m-1) \mu \right).\\ \nonumber
\end{align}
where the last inequality follows from $\lambda \geq 0$ and $ \theta \geq \frac{1}{3}.$\\
{\bf Case 2.2.1:} If $ \mu \geq 0$ then from \eqref{eq:up}
$$ z_{\sf Aff}({\cal U}) \geq  \frac{r}{3\sqrt{m}} \geq  \frac{m-\sqrt{m}}{3\sqrt{m}} \geq \frac{\sqrt{m}}{6}    \; \;  \text{for }  m \geq 4  \; \;  (\text{Contradiction with \eqref{contradic}})$$\\
{\bf Case 2.2.2:} Now suppose that $ \mu < 0$, by non-negativity of $ \mb y ( \mb \nu_1) $ we have
$$ \frac{r}{\sqrt{m}} \mu + \lambda \geq 0$$
i.e. $$ \mu \geq \frac{-\lambda \sqrt{m}}{r}  $$
and from \eqref{eq:up}
\begin{align*}
 z_{\sf Aff}({\cal U})  &\geq    \frac{r}{\sqrt{m}} \left(  \frac{1}{3} + (m-1) \mu \right) \\
 & \geq  \frac{r}{\sqrt{m}}\left(  \frac{1}{3} - \lambda \sqrt{m}\frac{m-1}{r}   \right) \\
& \geq  \frac{r}{\sqrt{m}}\left(  \frac{1}{3} - \frac{1}{12} \frac{m-1}{r}  \right)  \geq   \frac{r}{\sqrt{m}}  \left(  \frac{1}{3} - \frac{1}{6}   \right)   \; \; \text{for } m \geq 4  . \\
& \geq  \frac{\sqrt{m}}{12} \; \; (\text{Contradiction with \eqref{contradic}})\\
\end{align*}
We conclude that $ \alpha \geq \frac{1}{24 \sqrt{m}}$ and consequently
 $$z_{\sf Aff}({\cal U}) \geq  \mb c^T \mb x = \frac{\alpha m}{15} \geq \frac{\sqrt{m}}{360} =  \Omega ( \sqrt{m}).$$
 Hence,  $$z_{\sf Aff}({\cal U})= \Omega ( \sqrt{m}) \cdot  z_{\sf AR}({\cal U}).$$
 $\mb c^T \mb x= \mb c^T \mb x^*_{\sf Aff}$
Moreover, for any optimal affine solution, the cost of the first-stage affine solution  $\mb x^*_{\sf Aff}$ is $ \Omega ({\sqrt{m}})$ away from the optimal adjustable problem ~\eqref{eqn:ar}, i.e. $  \mb c^T \mb x^*_{\sf Aff} =\mb c^T \mb x =\Omega ( \sqrt{m})\cdot z_{\sf AR}({\cal U})$.
  \end{proof}
}
\section{Proof of Theorem \ref{thm:worstcase}}\label{apx-proofs:worstcase}

\begin{proof}

Let us find the order of the left hand side ratio in inequality \eqref{contr}. We have,
\begin{align*}
\frac{\binom{\sqrt{m}}{m^{\epsilon}} \cdot \binom{m-m^{\epsilon}}{\sqrt{m}-m^{\epsilon}} }{\binom{m}{\sqrt{m}}}&=
\frac{  (\sqrt{m})!    \times  (m-{m^{\epsilon}})! \times (m-\sqrt{m})!  \times   (\sqrt{m})!   }{ (    \sqrt{m} -   m^{\epsilon} )!\times (      m^{\epsilon} )! \times    m! \times (    \sqrt{m} -   m^{\epsilon} )! \times(m-\sqrt{m})!    } \\
&= \left(\frac{  (\sqrt{m})!   }{ (    \sqrt{m} -   m^{\epsilon} )! } \right)^2 \cdot \frac{     (m-{m^{\epsilon}})! \  }{ (      m^{\epsilon} )! \times    m!   } .\\
\end{align*}
By Stirling's approximation, we have
\begin{align*}
 \left(\sqrt{m}\right)! &= \Theta \left(   m^{\frac{1}{4}}  \left(\frac{\sqrt{m}}{e}   \right)^{\sqrt{m}}                \right). \\
 \left(\sqrt{m}-m^{\epsilon}\right)! &= \Theta \left(  ( \sqrt{m}-m^{\epsilon})^{\frac{1}{2}}  \left(\frac{\sqrt{m}-m^{\epsilon}}{e}   \right)^{\sqrt{m}-m^{\epsilon}}                \right). \\
 \left(m-m^{\epsilon}\right)! &= \Theta \left(  ( m-m^{\epsilon})^{\frac{1}{2}}  \left(\frac{m-m^{\epsilon}}{e}   \right)^{m-m^{\epsilon}}                \right). \\
\left(m\right)! &= \Theta \left(   m^{\frac{1}{2}}  \left(\frac{m}{e}   \right)^{m}                \right) .\\
\left(m^{\epsilon}\right)! &= \Theta \left(   m^{\frac{1}{2} \epsilon }\left(\frac{m^{\epsilon}}{e}   \right)^{m^{\epsilon}}                \right) .
\end{align*}
All together,
$$\frac{\binom{\sqrt{m}}{m^{\epsilon}} \cdot \binom{m-m^{\epsilon}}{\sqrt{m}-m^{\epsilon}} }{\binom{m}{\sqrt{m}}}= \Theta \left( \frac {   \left( \sqrt{m} \right)^{2 \sqrt{m}} \cdot  \left( m-m^{\epsilon}    \right)^{ \left(m-m^{\epsilon}\right)}  }{  m^{\frac{1}{2} \epsilon }  \cdot
\left( \sqrt{m}-m^{\epsilon}    \right)^{2 \left(\sqrt{m}-m^{\epsilon}\right)}
 \cdot   m^m \cdot m^{ \epsilon m^{\epsilon}}
} \right). $$
We have
$$ \left( m-m^{\epsilon}    \right)^{ \left(m-m^{\epsilon}\right)} = \Theta \left(  m ^{ \left(m-m^{\epsilon}\right)}  \cdot e^{-m^{\epsilon}+ \frac{m^{2\epsilon}}{m}} \right),$$ and
$$ \left( \sqrt{m}-m^{\epsilon}    \right)^{ 2 \left(\sqrt{m}-m^{\epsilon}\right)} =   \Theta \left( \left(\sqrt{m} \right)^{2 \left(\sqrt{m}-m^{\epsilon}\right)}  \cdot e^{-  2m^{\epsilon}+ 2 \frac{m^{2\epsilon}}{\sqrt{m}}}\right),$$
WLOG, we can suppose that $ \epsilon < \frac{1}{4}$, therefore
\begin{align*}
\frac{\binom{\sqrt{m}}{m^{\epsilon}} \cdot \binom{m-m^{\epsilon}}{\sqrt{m}-m^{\epsilon}} }{\binom{m}{\sqrt{m}}} &= \Theta \left( \frac {
e^{ m^{\epsilon} - 2 \frac{m^{2\epsilon}}{\sqrt{m}} + \frac{m^{2\epsilon}}{m}   } }{ m^{ \epsilon m^{\epsilon} +\frac{1}{2}\epsilon}
} \right) \\
&=\Theta \left( \frac {
e^{ m^{\epsilon}    } }{ m^{ \epsilon m^{\epsilon} +\frac{1}{2}\epsilon}
} \right). \\
\end{align*}

We have,
$$\Theta \left( \frac {    Q(m)e^{ m^{\epsilon}    } }{ m^{ \epsilon m^{\epsilon} +\frac{1}{2}\epsilon}} \right) \geq 1, $$
but the later inequality contradicts $$ \lim_{m\to\infty} \frac {    Q(m)e^{ m^{\epsilon}    } }{ m^{ \epsilon m^{\epsilon} +\frac{1}{2}\epsilon}} = 0.$$
 \end{proof}

\red{
\section{Domination for non-permutation invariant sets}\label{apx-:prop:inters}
\begin{prop} \label{prop:inters}
Suppose Algorithm \ref{algo2} returns $\beta$ and $\mb v$ for some uncertainty set ${\cal U}$. Then the set \eqref{newdom} is a dominating set for $\cal U$.
\end{prop}
\begin{proof}
Suppose Algorithm \ref{algo2} returns $\beta$ and $\mb v$,then the inequality \eqref{ineq:beta-v} is verified, namely,
$$\frac{1}{\beta} \sum_{i=1}^m \left( h_i -\beta v_i  \right)^{+} \leq 1, \; \forall \mb h \in {\cal U}.$$
Recall the dominating point \eqref{eq:mapping}
$$\mb{ \hat{h}}(\mb h) = \beta \mb v + ( \mb h - \beta \mb v )_+.$$
We have
$$ \mb{ \hat{h}}(\mb h) = \beta \left(    \sum_{i=1}^m \frac{( h_i -\beta v_i )^{+} }{\beta} ( \mb e_i + \mb v )  + \underbrace{\left( 1- \sum_{i=1}^m \frac{( h_i -\beta v_i )^{+} }{\beta}  \right) }_{ \geq 0} \mb v  \right)  \in {\hat{\cal U}}  $$
where
$${\hat{\cal U}} = \beta \cdot {\sf conv} \left( \mb v, \mb e_1 + \mb v, \ldots, \mb e_m + \mb v \right)$$
Hence ${\hat{\cal U}}$ is a dominating set.
\end{proof}
}

\red{
\section{Domination for the generalized budget set}\label{apx-prop:poly}
\begin{prop} \label{prop:poly}
Let consider
\begin{equation}\label{domset:poly}
{\hat{\cal U}}= {\sf conv} \left( \mb e_1, \ldots, \mb e_m, \frac{1}{m-1-2\theta} \mb e \right)
\end{equation}
The set \eqref{domset:poly} dominates the uncertainty set \eqref{set:polytope}.
\end{prop}
\begin{proof}
Consider the uncertainty set \eqref{set:polytope} given by
$${\cal U}= \left\{ \mb h \in [0,1]^m \; \Bigg\vert \; \sum_{i=1}^m h_i \leq 1 + \theta ( h_i +h_j) \; \; \forall i \neq j \right\}$$
and
$${\hat{\cal U}}= {\sf conv} \left( \mb e_1, \ldots, \mb e_m, \frac{1}{m-1-2\theta} \mb e \right).$$
Note that in our setting we choose $ \theta > \frac{m-1}{2}$.
Take any $\mb h \in {\cal U}$. Suppose WLOG that
$$ h_1 \leq h_2 \leq \ldots \leq h_m $$
Hence, by definition of ${\cal U}$
$$ \mb e^T \mb h \leq 1+ \theta ( h_1+h_2) $$
To prove that $\hat{ \cal U}$ dominates ${ \cal U}$, it is sufficient to find $ \alpha_1, \alpha_2,\ldots,\alpha_{m+1}$ non-negative reals with $\sum_{i=1}^{m+1} \alpha_i \leq 1$ such that for all $ i \in [m],$ $$ \; h_i \leq  \alpha_i + \frac{1}{m-1-2\theta}\alpha_{m+1}.$$
We choose $\alpha_{m+1} = (m-1-2\theta) \cdot \frac{h_1+h_2}{2} $, $\alpha_1=h_1$ and   for $i \geq 2$, $\alpha_i = h_i -  \frac{h_1+h_2}{2}$.
We can verify that
$$ \alpha_1+  \frac{1}{m-1-2\theta} \alpha_{m+1}  \geq  \alpha_1 =h_1 $$
and for $ i \geq 2$,
$$ \alpha_i +  \frac{1}{m-1-2\theta} \alpha_{m+1}  =h_i $$
Moreover, $\alpha_{m+1} \geq 0$, $\alpha_1 \geq 0$ and for $i \geq 2$, $\alpha_i \geq 0$ since $h_1+h_2= min_{i\neq j} (h_i+h_j)$. Finally,
\begin{align*}
\sum_{i=1}^{m+1} \alpha_i &= \sum_{i=1}^{m} h_i - (m-1)\cdot \frac{h_1+h_2}{2}+ (m-1-2\theta) \cdot \frac{h_1+h_2}{2}\\
& \leq 1+ \theta ( h_1+h_2) - (m-1)\cdot \frac{h_1+h_2}{2}+ (m-1-2\theta) \cdot \frac{h_1+h_2}{2} =1.
\end{align*}
Note that the construction of this dominating set is slightly different from the general approach in Section 3 since we do not scale the unit vectors $\mb e_i$ in ${\hat{\cal U}}$.
\end{proof}
}

\end{appendix}

\end{document}